\documentclass[11pt,reqno]{amsart}
\usepackage{amsmath,amssymb,mathrsfs}
\usepackage{graphicx,cite}
\usepackage{subfigure}
\usepackage{epstopdf}
\usepackage{graphics} 
\usepackage{epsfig} 
\usepackage{tikz}
\usepackage{paralist}
\usepackage{booktabs}
\usepackage{enumerate}
\usepackage{enumitem}
\usepackage{mdwlist}

\newtheorem{theorem}{Theorem}[section]

\newtheorem{remark}[theorem]{Remark}

\numberwithin{equation}{section}
\begin{document}

\title[Biharmonic wave scattering]{A novel boundary integral formulation for the biharmonic wave scattering problem}

\author{Heping Dong}
\address{School of Mathematics, Jilin University, Changchun,  Jilin 130012, China}
\email{dhp@jlu.edu.cn}

\author{Peijun Li}
\address{Department of Mathematics, Purdue University, West Lafayette, Indiana 47907, USA}
\email{lipeijun@math.purdue.edu}

\thanks{The first author was supported in part by the NSFC grant 12171201. The second author was supported partially by the NSF grant DMS-2208256.}



\begin{abstract}
This paper is concerned with the cavity scattering problem in an infinite thin plate, where the out-of-plane displacement is governed by the two-dimensional biharmonic wave equation. Based on an operator splitting, the scattering problem is recast into a coupled boundary value problem for the Helmholtz and modified Helmholtz equations. A novel boundary integral formulation is proposed for the coupled problem. By introducing an appropriate regularizer, the well-posedness is established for the system of boundary integral equations. Moreover, the convergence analysis is carried out for the semi- and full-discrete schemes of the boundary integral system by using the collocation method. Numerical results show that the proposed method is highly accurate for both smooth and nonsmooth examples. 
\end{abstract}

\maketitle

\section{Introduction}

Scattering problems for biharmonic waves have significant applications in thin plate elasticity. For instance, the ultra-broadband elastic cloaking has been experimentally realized at acoustic frequencies and could serve as a model for seismic cloaks to protect infrastructure from earthquakes \cite{FGE-09, SWW-12}. Similar to photonic and phononic crystals, platonic crystals (PlaCs) are fabricated by using regularly spaced arrays of holes, called cavities, cut into thin elastic plates \cite{HMMM-12, MMM-09}. The design of PlaCs can be utilized to guide or harness 
destructive wave energy for constructive purposes \cite{EP-07, S-13, WUW2004}. As an emerging passive vibration control technique, the acoustic black hole (ABH) has recently been developed for vibration and noise reduction of a system \cite{PGCS-20}. The basic idea of ABH is to arrange some inhomogeneous medium in a thin beam or plate so that the wave can be trapped from propagating. In the wave motion of these structures, the out-of-plane displacement of the thin plate satisfies the biharmonic wave equation. Although the scattering theory has been well developed for acoustic, elastic, and electromagnetic waves, many scattering problems remain unsolved for biharmonic waves. This paper concerns the biharmonic wave scattering of a time-harmonic plane wave by a cavity in an infinite thin plate. A novel boundary integral formulation is proposed for the boundary value problem and a spectrally accurate collocation method is developed for the boundary integral system. 

The method of boundary integral equations has several advantages: the solution is represented by surface distributions which leads to fewer unknowns; the radiation condition is exactly enforced that avoids the truncation error of artificial boundary conditions; the surface geometry can be represented to high precision using appropriate boundary elements. However, in a conventional boundary integral method, these advantages may be offset by two factors: the high cost of evaluating interactions among all the boundary elements and the need to evaluate singular integrals. The former can be resolved by fast evaluation methods such as the panel clustering algorithm \cite{HN-89} or the fast multipole method \cite{GR-87}, and the latter has to be treated using special quadrature formulas or analytical expansions. Despite the difficulties, the boundary integral method provides an efficient approach for solving many problems, especially for those imposed in unbounded domains \cite{DR-shu1, DR-shu2}. 

We consider the cavity scattering problem in an infinite thin elastic plate, where the wave motion is governed by the two-dimensional biharmonic wave equation. As an exterior boundary value problem, it is analogous to the obstacle scattering problem for the acoustic, elastic, or electromagnetic wave equation. Compared with the obstacle scattering problem, the cavity scattering problem is much less investigated. A numerical study was given in \cite{S-13} on the boundary integral method for solving the biharmonic wave scattering by smooth cavities embedded in thin plates. In \cite{SMM-11}, boundary integral equations were deduced for cavities with different boundary conditions, and the full discretization of the boundary integral system was discussed by expanding the unknown boundary values in terms of Fourier series. However, the mathematical and numerical analysis are not available on the method of boundary integral equations for solving the biharmonic wave scattering problem. 

Motivated by \cite{DLL2021}, we present the first theoretical analysis on the biharmonic wave scattering problem. Based on a splitting of the biharmonic wave operator, two auxiliary functions are introduced to represent the Helmholtz and modified Helmholtz wave components, respectively. The scattering problem is recast into a coupled boundary value problem for the Helmholtz and modified Helmholtz equations. The uniquneness is established for the scattering problem by utilizing Rellich's lemma and the exponentially decaying property of the modified Helmholtz wave component. By the double- and single-layer potentials representing the Helmholtz and modified Helmoltz wave components, respectively, a novel boundary integral formulation is proposed for the coupled boundary value problem. Following \cite{DLL2021, LL2019}, we introduce a regularizer to the boundary integral system and split the singular operator into an isomorphic operator plus a compact one. The Riesz--Fredholm theory is then employed to show the well-posedness of the coupled boundary integral system. Moreover, the collocation method is empolyed to examine the semi- and full-discretization of the boundary integral equations. Special quadratures are adopted for the hypersingular and logarithmic singular integrals to carry out the convergence analysis for both schemes. In \cite{KS1993, Kress1995}, the quadratures were investigated on the similar hypersingular and logarithmic singular integrals for solving the acoustic obstacle scattering problem. In \cite{Kress2014}, a collocation method was developed and the convergence was analyzed by using the trigonometric interpolation to discretize the principal part of the hypersingular operator. We refer to \cite{KirschRitter1999, MP-book1986, SV1996} for studies on general singular integral equations. In addition to the convergence analysis, numerical experiments are provided to demonstrate the high accuracy of the proposed method for both smooth and nonsmooth examples.

The rest of this paper is structured as follows. In section \ref{s_pf}, the scattering problem is introduced and the uniqueness is proved for the boundary value problem. Section \ref{s_bie} presents the boundary integral equations and discusses the well-posedness of the coupled system. The convergence analysis is carried out in section \ref{s_cm} for the semi- and full-discrete schemes of the system. Section \ref{s_ne} presents numerical experiments to verify the theoretical findings and to demonstrate the superior performance of the  proposed method. Section \ref{s_c} concludes the paper with remarks and future work. 

\section{Problem formulation}\label{s_pf}

We consider the biharmonic wave scattering problem in a two-dimensional thin plate with a cavity, which is represented by a bounded domain $D\subset\mathbb R^2$ with an analytic boundary $\Gamma$. The plate is assumed to be made of a homogeneous and isotropic elastic medium in $\mathbb R^2\setminus\overline D$. Let $\tau=(\tau_1, \tau_2)^\top$ and $\nu=(\nu_1, \nu_2)^\top$ be the tangential and normal vectors on $\Gamma$, respectively. It is clear to note  $\tau_1=-\nu_2$ and $\tau_2=\nu_1$. The plate is excited by a time-harmonic plane wave
\begin{equation}\label{ui}
	u^{\rm inc}(x)=e^{\mathrm{i}\kappa x\cdot d},\quad x\in\mathbb R^2,
\end{equation}
where $\kappa>0$ is the wavenumber and $d=(\cos\theta, \sin\theta)^\top$ is the incident direction with $\theta\in [0,2\pi)$ being the incident angle. 

The out-of-plane displacement of the plate $u$ satisfies the two-dimensional biharmonic wave equation
\begin{align}\label{uwe}
	\Delta^2u-\kappa^4 u=0 \quad {\rm in}\, \mathbb{R}^2\setminus \overline{D}. 
\end{align}
Assuming that the cavity satisfies the clamped boundary condition, we have 
\begin{align}\label{bc}
	u=\partial_\nu u=0 \quad {\rm on}\, \Gamma.
\end{align}
The total displacement $u$ is composed of the incident wave field $u^{\rm inc}$ and the scattered field $v$, i.e., 
\[
u= u^{\rm inc}+ v. 
\]
It can be verified from \eqref{ui}--\eqref{uwe} that the scattered field $v$ satisfies
\begin{align}\label{vwe}
	\Delta^2v-\kappa^4 v=0 \quad {\rm in}\, \mathbb{R}^2\setminus \overline{D}. 
\end{align}
By \eqref{bc}, $v$ satisfies the following boundary conditions on $\Gamma$:
\begin{align}\label{bcv}
	v=-u^{\rm inc}, \quad \partial_\nu v=- \partial_\nu u^{\rm inc}. 
\end{align}
In addition, $v$ and $\Delta v$ are required to satisfy the Sommerfeld radiation condition
\begin{align}\label{sc}
	\lim_{r\to\infty}r^{\frac{1}{2}}(\partial_r
	v-\mathrm{i}\kappa v)=0, \quad
	\lim_{r\to\infty}r^{\frac{1}{2}}(\partial_r
	\Delta v-\mathrm{i}\kappa \Delta v)=0, \quad r=|x|.
\end{align}

Consider two auxiliary functions
\begin{align}\label{vhvm}
	v_{\rm H} = -\frac{1}{2\kappa^2}(\Delta v - \kappa^2 v), \quad v_{\rm M} = \frac{1}{2\kappa^2}(\Delta v + \kappa^2 v).
\end{align}
It is clear to note $v=v_{\rm H}+v_{\rm M}$ and $\Delta v=\kappa^2(v_{\rm M}-v_{\rm H})$. From \eqref{vwe}, we have
\[
	(\Delta - \kappa^2)(\Delta + \kappa^2)v=0 \quad {\rm in}\, \mathbb{R}^2\setminus \overline{D},
\]
which implies that $v_{\rm H}$ and $v_{\rm M}$ satisfy the Helmholtz equation and the modified Helmholtz equation in $\mathbb{R}^2\setminus \overline{D}$, respectively, i.e., 
\begin{align} \label{vHvM}
	\Delta v_{\rm H} + \kappa^2 v_{\rm H} = 0, \quad \Delta v_{\rm M} - \kappa^2 v_{\rm M} = 0. 
\end{align}
By \eqref{bcv}, $v_{\rm H}$ and $v_{\rm M}$ satisfy the coupled boundary conditions on $\Gamma$:
\begin{align}\label{bccouple}
	v_{\rm H}+v_{\rm M}= f_1,\quad \partial_\nu v_{\rm H} + \partial_\nu v_{\rm M}= f_2, 
\end{align}
where $f_1=-u^{\rm inc}$ and $f_2=-\partial_\nu u^{\rm inc}$. Combining \eqref{sc} and \eqref{vhvm} yields that $v_{\rm H}$ and $v_{\rm M}$ satisfy the Sommerfeld radiation condition 
\begin{align}\label{sc1}
	\lim_{r\to\infty}r^{\frac{1}{2}}(\partial_r
	v_{\rm H}-\mathrm{i}\kappa v_{\rm H})=0, \quad \lim_{r\to\infty}r^{\frac{1}{2}}(\partial_r
	v_{\rm M}-\mathrm{i}\kappa v_{\rm M})=0, \quad r=|x|.
\end{align}

Clearly, the scattering problems \eqref{vwe}--\eqref{sc} and \eqref{vHvM}--\eqref{sc1} are equivalent. Hence, it suffices to consider the scattering problem \eqref{vHvM}--\eqref{sc1}, where $v_{\rm H}$ and $v_{\rm M}$ are called the Helmholtz and modified Helmholtz wave components, respectively. 

\begin{remark}\label{vMdecay}
	Let $v_{\rm M}$ be a solution to the scattering problem \eqref{vHvM}--\eqref{sc1}. Denote $D\subset B_a=\{x\in\mathbb R^2: |x|<a\}$. Then $v_{\rm M}$ has the Fourier series expansion
	\[
		v_{\rm M}(r,\theta)=\sum_{n=-\infty}^{\infty}\frac{K_n(\kappa r)}{K_n(\kappa a)}\hat{v}^{(n)}_{\rm M}(a)e^{\mathrm{i}n\theta}, \quad r=|x|>a, 
	\]
	where $\hat{v}^{(n)}_{\rm M}(a)$ are the Fourier coefficients of $v_{\rm M}$ on $\partial B_a=\{x\in\mathbb R^2: |x|=a\}$ given by 
	$$
	\hat{v}^{(n)}_{\rm M}(a)=\frac{1}{2\pi}\int_{0}^{2\pi}v_{\rm M}(a,\theta)e^{-\mathrm{i}n\theta}\mathrm{d}\theta
	$$
	and $K_n(z)$ denotes the modified Bessel function of the second kind with order $n$. By \cite[p. 511]{Watson1922}, $K_n(z)$ has no zeros for $|\arg z|\leq\pi/2$, which implies that $K_n(\kappa a)\neq0$. Moreover, we have from \cite[\S 3.7 and \S 7.23]{Watson1922} that $K_n(t)=K_{-n}(t)$ and it admits the asymptotic expansion
\[
		K_n(t)=\sqrt{\frac{\pi}{2 t}}e^{-t}\Big\{1+O\big(\frac{1}{t}\big)\Big\},\quad t\to +\infty, \quad t>0, 
\]
	which shows that $v_{\rm M}$ and $\partial_r v_{\rm M}$ decay exponentially as $r\to\infty$ for a fixed $\kappa$.
\end{remark}

\begin{theorem}\label{unique1}
	The boundary value problem \eqref{vHvM}--\eqref{sc1} has at most one solution.  
\end{theorem}

\begin{proof}
	It is required to show that $v_{\rm H}=v_{\rm M}=0$ in $\mathbb{R}^2\setminus\overline{D}$ when $f_1=f_2=0$. Applying Green's theorem in $D_a:= B_a\setminus\overline D$ and the boundary condition \eqref{bccouple}, we obtain 
	\begin{align*}
		\int_{\partial B_a}v_{\rm H}\partial_\nu\overline{v_{\rm H}}\mathrm{d}s&=\int_{D_a}\big(v_{\rm H}\Delta\overline{v_{\rm H}}+\nabla v_{\rm H}\cdot\nabla\overline{v_{\rm H}}\big)\mathrm{d}x+\int_{\Gamma}v_{\rm H}\partial_\nu\overline{v_{\rm H}}\mathrm{d}s\nonumber\\
		&=\int_{D_a}\big(-\kappa^2|v_{\rm H}|^2 + |\nabla v_{\rm H}|^2\big)\mathrm{d}x+\int_{\Gamma}v_{\rm M}\partial_\nu\overline{v_{\rm M}}\mathrm{d}s,\\
		\int_{\partial B_a}v_{\rm M}\partial_\nu\overline{v_{\rm M}}\mathrm{d}s&=\int_{D_a}\big(\kappa^2|v_{\rm M}|^2 + |\nabla v_{\rm M}|^2\big)\mathrm{d}x+\int_{\Gamma}v_{\rm M}\partial_\nu\overline{v_{\rm M}}\mathrm{d}s,
	\end{align*}
	which give
	\[
		\Im \int_{\partial B_a}v_{\rm H}\partial_\nu\overline{v_{\rm H}}\mathrm{d}s=\Im \int_{\Gamma}v_{\rm M}\partial_\nu\overline{v_{\rm M}}\mathrm{d}s=\Im \int_{\partial B_a}v_{\rm M}\partial_\nu\overline{v_{\rm M}}\mathrm{d}s.
	\]
	It follows from the above equation and the Sommerfeld radiation condition \eqref{sc} that
	\begin{align}\label{decay1}
		\lim_{a\to\infty}\int_{\partial B_a}|\partial_\nu v_{\rm H}-\mathrm{i}\kappa v_{\rm H}|^2\mathrm{d}s&=\lim_{a\to\infty}\int_{\partial B_a}\big\{|\partial_\nu v_{\rm H}|^2+\kappa^2|v_{\rm H}|^2+2\kappa\Im(v_{\rm H}\partial_\nu\overline{v_{\rm H}})\big\}\mathrm{d}s\nonumber\\
		&=\lim_{a\to\infty}\int_{\partial B_a}\big\{|\partial_\nu v_{\rm H}|^2+\kappa^2|v_{\rm H}|^2+2\kappa\Im(v_{\rm M}\partial_\nu\overline{v_{\rm M}})\big\}\mathrm{d}s=0.
	\end{align}
	By Remark \ref{vMdecay} and \eqref{decay1}, we have
    \[
		\lim_{a\to\infty}\int_{\partial B_a}\big\{|\partial_\nu v_{\rm H}|^2+\kappa^2|v_{\rm H}|^2\big\}\mathrm{d}s=0.
	\]
	Using Rellich's lemma yields $v_{\rm H}=0$ in $\mathbb{R}^2\setminus \overline{D}$. It follows from the continuity of $v_{\rm H}$ and $\partial_\nu v_{\rm H}$ and \eqref{bccouple} that $v_{\rm H}=\partial_\nu v_{\rm H}=0$ and $v_{\rm M}=\partial_\nu v_{\rm M}=0$ on $\Gamma$, which shows $v_{\rm M}=0$ in $\mathbb{R}^2\setminus \overline{D}$ by Green's representation theorem (cf. \cite[Theorem 2.5]{DR-shu2}).
\end{proof}

\section{Boundary integral formulation}\label{s_bie}

In this section, we propose a novel formulation of boundary integral equations for the scattering problem \eqref{vHvM}--\eqref{sc1}. Based on the Riesz--Fredholm theory, a regularizer is constructed to split the operator equation into the form of an isomorphic operator plus a compact one.

\subsection{Coupled integral equations}

Denote by $G_{\rm H}$ and $G_{\rm M}$ the fundamental solutions to the two-dimensional Helmholtz and modified Helmholtz equations, respectively. Explicitly, we have 
\[
	G_{\rm H}(x,y)=\frac{\mathrm{i}}{4}H_0^{(1)}(\kappa|x-y|),\quad 	G_{\rm M}(x,y)=\frac{\mathrm{i}}{4}H_0^{(1)}(\mathrm{i}\kappa|x-y|),
\]
where $H_0^{(1)}$ is the Hankel function of the first kind with order zero. Using the fundamental solutions, we represent the solutions of \eqref{vHvM} as the double- and single-layer potentials with densities $g_1$ and $g_2$, respectively, i.e.,
\begin{align}\label{dslayer}
	v_{\rm H}(x)=\int_\Gamma \frac{\partial G_{\rm H}(x,y)}{\partial\nu(y)}g_1(y)\mathrm{d}s(y),\quad v_{\rm M}(x)=\int_\Gamma G_{\rm M}(x,y)g_2(y)\mathrm{d}s(y),\quad x\in\mathbb{R}^2\setminus\Gamma_D. 
\end{align}

It follows from the solution representation \eqref{dslayer}, the boundary conditions \eqref{bcv}, and the jump relations of the single- and double-layer potentials, we may deduce the boundary integral equations on $\Gamma$: 
\begin{align}
	\begin{split}\label{cboundaryIE}
		\frac{1}{2}g_1(x)+\int_{\Gamma}\frac{\partial G_{\rm H}(x,y)}{\partial\nu(y)}g_1(y)\,\mathrm{d}s(y)+ \int_{\Gamma}G_{\rm M}(x,y)g_2(y)\,\mathrm{d}s(y)&=-u^{\rm inc}(x),
		\\
		\int_{\Gamma}\frac{\partial^2 G_{\rm H}(x,y)}
		{\partial\nu(x)\partial\nu(y)}g_1(y)\,\mathrm{d}s(y)-\frac{1}{2}g_2(x)+\int_{\Gamma}\frac{\partial G_{\rm M}(x,y)}
		{\partial\nu(x)}g_2(y)\,\mathrm{d}s(y)&=-\frac{\partial u^{\rm inc}(x)}{\partial\nu(x)}.
	\end{split}
\end{align}
For $x\in\Gamma$, introduce the following integral operators:
\begin{align*}
	&(\mathcal{L}g)(x)=2\int_{\Gamma}\frac{\partial G_{\rm H}(x,y)}
	{\partial\nu(y)}g(y)\mathrm{d}s(y), 
	&&(\mathcal{T}g)(x)=2\int_{\Gamma_D}\frac{\partial^2 G_{\rm H}(x,y)}
	{\partial\nu(x)\partial\nu(y)}g(y)\mathrm{d}s(y),\\
	&(\mathcal{K}g)(x)=2\int_{\Gamma_D}\frac{\partial G_{\rm M}(x,y)}
	{\partial\nu(x)}g(y)\mathrm{d}s(y), 
	&&(\mathcal{S}g)(x)=2\int_{\Gamma}G_{\rm M}(x,y)g(y)\mathrm{d}s(y). 
\end{align*}
Since the fundamental solution $G_{\rm M}$ decays exponentially, the modified Helmholtz wave component $v_{\rm M}$ given in \eqref{dslayer} does not propagate. Hence we may obtain from \eqref{dslayer} that the far-field pattern of the scattered wave $v$ is  
\begin{align}\label{singlelayer_far}
	v_\infty(\hat{x})=\varrho\int_{\Gamma}\frac{\partial e^{-\mathrm{i}
			\kappa\hat{x}\cdot y}}{\partial\nu(y)} g_1(y)\,\mathrm{d}s(y),
\end{align}
where $\varrho= e^{\mathrm{i}\pi/4}/{\sqrt{8\kappa\pi}}$ and $\hat x=x/r$ is the observation direction.

With the help of the integral operators, it is convenient to rewrite \eqref{cboundaryIE} into the operator form
\begin{equation}\label{direct field}
	\begin{cases}
		g_1 + \mathcal{L} g_1+\mathcal{S} g_2=2f_1,\\
		\mathcal{T} g_1 - g_2 + \mathcal{K} g_2=2f_2. 
	\end{cases}
\end{equation}

\begin{theorem}\label{unique2}
	The coupled boundary integral equations \eqref{direct field} has at most one solution provided that the wavenumber $\kappa$ is not the Dirichlet eigenvalue of the Helmholtz equation in $D$. 
\end{theorem}

\begin{proof}
	It is required to show that $g_1=g_2=0$ if $f_1=f_2=0$. Let
	\begin{eqnarray*}
		v_{\rm H}(x)=\begin{cases}
			v_{\rm H}^i, &  x\in D,\\
			v_{\rm H}^e, &  x\in \mathbb{R}^2\setminus \overline{D},
		\end{cases}\qquad 
		v_{\rm M}(x)=\begin{cases}
			v_{\rm M}^i, &  x\in D,\\
			v_{\rm M}^e, &  x\in \mathbb{R}^2\setminus \overline{D}.
		\end{cases}
	\end{eqnarray*}
	Since $v_{\rm H}^e$ and $v_{\rm M}^e$ satisfy the boundary value problem \eqref{vHvM}--\eqref{sc1}, it follows from 
	Theorem \ref{unique1} that they are identically zero in $\mathbb{R}^2\setminus \overline{D}$. Using the jump relations of the single- and double-layer potentials, we have on $\Gamma$ that 
	\begin{eqnarray*}
		&v_{\rm H}^e - v_{\rm H}^i = g_1, \quad v_{\rm M}^e - v_{\rm M}^i = 0,\\
		&\partial_\nu v_{\rm H}^e - \partial_\nu v_{\rm H}^i = 0, \quad \partial_\nu v_{\rm M}^e - \partial_\nu v_{\rm M}^i = -g_2.
	\end{eqnarray*}
	Since neither $\kappa$ nor $\mathrm{i}\kappa$ is the Dirichlet eigenvalue of the Helmholtz equation in $D$, we deduce $v_{\rm H}^i=v_{\rm M}^i=0$ in $D$, which implies $g_1=g_2=0$.
\end{proof} 

By Theorem \ref{unique2}, the uniqueness holds for the coupled system \eqref{direct field} if the wavenumber $\kappa$ is not the Dirichlet eigenvalue of the Helmholtz equation in $D$. In general, this assumption can be removed by using combined double- and single-layer potentials for the boundary integral formulation \cite{DR-shu1}. In this work, we assume that this assumption is satisfied and the coupled system \eqref{direct field} has at most one solution. 

\subsection{Parameterization}

Let the analytic boundary $\Gamma$ have a parametrization of the form
\[
	\Gamma=\{\gamma(t)=(\gamma_1(t),\gamma_2(t))\in\mathbb R^2: 0\leq t<2\pi\}, 
\]
where $\gamma_1$ and $\gamma_2$ are $2\pi$-periodic functions satisfying $|\gamma'(t)|>0$ for $t\in[0, 2\pi)$. Plugging the parameterization of $\Gamma$ into $\mathcal{L}, \mathcal{S}, \mathcal{K}$, and $\mathcal{T}$, we define the parameterized integral operators 
\begin{align*}
	&(L\phi)(t)=\int_0^{2\pi}l(t,\zeta)\phi(\zeta)\mathrm{d}\zeta, 
	&&(K\phi)(t)=\int_0^{2\pi}k(t,\zeta)\phi(\zeta)\mathrm{d}\zeta, \\
	&(S\phi)(t)=\int_0^{2\pi}s(t,\zeta)
	\phi(\zeta)\mathrm{d}\zeta,
	&&(R\phi)(t)=\int_0^{2\pi}r(t,\zeta)
	\phi(\zeta)\mathrm{d}\zeta,\\
	&(H\phi)=\int_0^{2\pi}h(t,\zeta)
	\phi(\zeta)\mathrm{d}\zeta,
\end{align*}
where the integral kernels are given by 
\begin{align*}
	l(t,\zeta)&=\frac{\mathrm{i}\kappa}{2}n(\zeta)\cdot[\gamma(t)-\gamma(\zeta)] \frac{H_1^{(1)}(\kappa|\gamma(t)-\gamma(\zeta)|)}{|\gamma(t)-\gamma(\zeta)|},\quad
	s(t,\zeta)=\frac{\mathrm{i}}{2}H_0^{(1)}(\mathrm{i}\kappa|\gamma(t)-\gamma(\zeta)|),\\
	k(t,\zeta)&=\frac{\kappa}{2}n(t)\cdot[\gamma(t)-\gamma(\zeta)] \frac{H_1^{(1)}(\mathrm{i}\kappa|\gamma(t)-\gamma(\zeta)|)}{|\gamma(t)-\gamma(\zeta)|},\\
	r(t,\zeta)&=\frac{\mathrm{i}\kappa^2}{2}H_0^{(1)}(\kappa|\gamma(t)-\gamma(\zeta)|)n(t)\cdot n(\zeta),\\
	h(t,\zeta)&=\frac{\mathrm{i}}{2}\tilde{h}(t,\zeta)\Big\{\kappa^2H_0^{(1)}(\kappa|\gamma(t)-\gamma(\zeta)|)- \frac{2\kappa H_1^{(1)}(\kappa|\gamma(t)-\gamma(\zeta)|)}{|\gamma(t)-\gamma(\zeta)|}\Big\}\\
	&\qquad+\frac{\mathrm{i}\kappa \gamma'(t)\cdot \gamma'(\zeta)}{2|\gamma(t)-\gamma(\zeta)|}H_1^{(1)}(\kappa|\gamma(t)-\gamma(\zeta)|)+\frac{1}{4\pi}\frac{1}{\sin^2\frac{1}{2}(t-\zeta)}.
\end{align*}
Here $n(t):=\big(\gamma'_2(t), -\gamma'_1(t)\big)^\top$, $n^\perp(t):=\gamma'(t)=\big(\gamma'_1(t), \gamma'_2(t)\big)^\top$, and
\[
\tilde{h}(t,\zeta)=\frac{\big[\gamma'(t)\cdot (\gamma(t)-\gamma(\zeta))\big]\big[\gamma'(\zeta)\cdot(\gamma(t)-\gamma(\zeta))\big]}{|\gamma(t)-\gamma(\zeta)|^2}.
\]
The kernels $l(t,\zeta)$, $s(t,\zeta)$, $k(t,\zeta)$, $r(t,\zeta)$, and $h(t,\zeta)$ can be decomposed into the following forms: 
\begin{align}\label{kerdecom_chi}
	\begin{split}
		\chi(t,\zeta)&= \chi_1(t,
		\zeta)\ln\Big(4\sin^2\frac{t-\zeta}{2}\Big)+\chi_2(t,\zeta), \\
		\chi_2(t,\zeta)&=\chi(t,\zeta)-\chi_1(t,
		\zeta)\ln\Big(4\sin^2\frac{t-\zeta}{2}\Big),
	\end{split}
\end{align}
where $\chi=l,s,k,r,h$, and $\chi_1, \chi_2$ are analytic functions. Explicitly, we have 
\begin{align*}
	l_1(t,\zeta)&= \frac{\kappa}{2\pi}n(\zeta)\cdot\big[
	\gamma(\zeta)-\gamma(t)\big]\frac{J_1(\kappa|\gamma(t)-\gamma(\zeta)|)}{
		|\gamma(t)-\gamma(\zeta)|}, ~~
	s_1(t,\zeta)= -\frac{1}{2\pi}J_0(\mathrm{i}\kappa|\gamma(t)-\gamma(\zeta)|), \\
	k_1(t,\zeta)&= \frac{\mathrm{i}\kappa}{2\pi}n(t)\cdot\big[
	\gamma(t)-\gamma(\zeta)\big]\frac{J_1(\mathrm{i}\kappa|\gamma(t)-\gamma(\zeta)|)}{
		|\gamma(t)-\gamma(\zeta)|}, \\
	r_1(t,\zeta)&= -\frac{\kappa^2}{2\pi}J_0(\kappa|\gamma(t)-\gamma(\zeta)|)n(t)\cdot n(\zeta), \\
	h_1(t,\zeta)&=-\frac{1}{2\pi}\tilde{h}(t,\zeta)\Bigg(\kappa^2J_0(\kappa|\gamma(t)-\gamma(\zeta)|)-\frac{2\kappa J_1(\kappa|\gamma(t)-\gamma(\zeta)|)}{
		|\gamma(t)-\gamma(\zeta)|}\Bigg)\\
	&\qquad-\frac{\kappa \gamma'(t)\cdot \gamma'(\zeta)}{2\pi|\gamma(t)-\gamma(\zeta)|}J_1(\kappa|\gamma(t)-\gamma(\zeta)|),
\end{align*}
where $J_0$ and $J_1$ are the Bessel functions of the first kind with order zero and one, respectively. The diagonal entries are given by 
\begin{align*}
	l_1(t,t)&=0, \quad l_2(t,t)=
	\frac{1}{2\pi}\frac{{n}(t)\cdot \gamma''(t)}{|\gamma'(t)|^2},\\
	s_1(t,t)&=-\frac{1}{2\pi}, \quad  s_2(t,t)=
	\frac{\mathrm{i}}{2}-\frac{C}{\pi}-\frac{1}{\pi}\ln\big(\frac{\mathrm{i}\kappa}{2}|\gamma'(t)|\big),\\
	r_1(t,t)&=-\frac{\kappa^2}{2\pi}|\gamma'(t)|^2, \quad  r_2(t,t)=\kappa^2|\gamma'(t)|^2
	\Big\{\frac{\mathrm{i}}{2}-\frac{C}{\pi}-\frac{1}{\pi}\ln\big(\frac{\kappa}{2}|\gamma'(t)|\big)\Big\},\\
		k_1(t,t)&=0, \quad k_2(t,t)=
	\frac{1}{2\pi}\frac{{n}(t)\cdot \gamma''(t)}{|\gamma'(t)|^2},\quad
	h_1(t,t)=-\frac{\kappa^2|\gamma'(t)|^2}{4\pi}, \\
	h_2(t,t)&=\Big(\pi\mathrm{i}-1-2C-2\ln\frac{\kappa|\gamma'(t)|}{2}\Big)\frac{\kappa^2|\gamma'(t)|^2}{4\pi}+\frac{1}{12\pi}\\
	&\qquad+\frac{[\gamma'(t)\cdot \gamma''(t)]^2}{2\pi|\gamma'(t)|^4}-\frac{|\gamma''(t)|^2}{4\pi|\gamma'(t)|^2}-\frac{\gamma'(t)\cdot \gamma'''(t)}{6\pi|\gamma'(t)|^2},
\end{align*}
where $C$ denotes Euler's constant. We refer to \cite{DR-shu2, Kress1995} for the details of the decomposition.

By the identities (cf. \cite[$(2.4)$ and $(2.6)$]{Kress1995})
\[
(\mathcal{T}g)(x)=2\int_{\Gamma}\frac{\partial\Phi(x,y)}{\partial\tau(x)}\frac{\partial g(y)}{\partial\tau}\mathrm{d}s(y)+2\kappa^2\int_{\Gamma}\Phi(x,y)\nu(x)\cdot\nu(y)g(y)\mathrm{d}s(y)
\]
and
\begin{align*}
&2\int_{\Gamma}\frac{\partial\Phi(x(t),y)}{\partial\tau(x(t))}\frac{\partial g(y)}{\partial\tau}\mathrm{d}s(y)\\
&=\frac{1}{|\gamma'(t)|}\int_0^{2\pi}\left(\frac{1}{2\pi}\cot\big(\frac{\zeta-t}{2}\big) \frac{\partial g(y(\zeta))}{\partial\tau}-h(t,\zeta)g(y(\zeta))\right) d\zeta,
\end{align*}
the parameterized integral operator $T$ can be represented as
\[
	|\gamma'|T\phi=(T_0\phi-M_0\phi)-H\phi+R\phi,
\]
where the operator $T_0: H^p[0,2\pi]\rightarrow
H^{p-1}[0,2\pi]$ and $M_0$ are defined by 
$$
(T_0\phi)(t)=\frac{1}{2\pi}\int_0^{2\pi}\cot\big(\frac{\zeta-t}{2}\big)\phi'(\zeta)\mathrm{d}\zeta+\frac{\mathrm{i}}{2\pi}\int_{0}^{2\pi}\phi(\zeta)\,\mathrm{d}\zeta\overset{\text{def}}{=}(\widetilde{T}_0\phi)(t)+(M_0\phi)(t).
$$
Substituting Maue's formula into \eqref{direct field} and multiplying $|\gamma'|$ on both sides of the second equation of \eqref{direct field}, we obtain  
\begin{align}\label{direct parafield}
	\Bigg(\left[     
	\begin{array}{cc}
		I     & 0 \\ 
		T_0   & -I
	\end{array}
	\right]+\left[     
	\begin{array}{cc}
		L      & S \\ 
		R-H-M_0   & K
	\end{array}
	\right]\Bigg)\left[                  
	\begin{array}{c}
		\psi_1 \\ 
		\psi_2
	\end{array}
	\right]=\left[                  
	\begin{array}{c}
		\eta_1 \\ 
		\eta_2
	\end{array}
	\right],
\end{align}
where $\eta_1=2(f_1\circ \gamma)$, $\eta_2=2(f_2\circ \gamma)|\gamma'|$, $\psi_1=(g_1\circ \gamma)$, $\psi_2=(g_2\circ \gamma)|\gamma'|$, and $I$ is the identity operator.

\subsection{Operator equations}

For $p\geq0$, we denote the space of $2\pi$-periodic functions by $H^p[0,2\pi]$, which is equipped with the norm
$$
\|u\|_p^2:=\sum_{m=-\infty}^\infty(1+m^2)^p|\hat{u}_m|^2<\infty,
$$
where the Fourier coefficients $\hat{u}_m$ are given by 
$$
\hat{u}_m=\frac{1}{2\pi}\int_0^{2\pi}u(t)e^{-\mathrm{i}mt}
\,\mathrm{d}t, \quad m=0,\pm1,\pm2,\cdots
$$
Define the Sobolev space
\begin{align*}
	H^p[0,2\pi]^2&=\Big\{\eta=(\eta_1, \eta_2)^\top: \eta_j\in H^p[0,2\pi], j=1,2\Big\},
\end{align*}
which has the standard norm $\|\eta\|_p=\|\eta_1\|_p+\|\eta_2\|_p$.

Define the integral operators $S_0: H^p[0,2\pi]\rightarrow
H^{p+1}[0,2\pi]$, $S_1: H^p[0,2\pi]\rightarrow
H^{p+2}[0,2\pi]$, and $S_2: H^p[0,2\pi]\rightarrow
H^{p+3}[0,2\pi]$ by
\begin{align*}
	(S_0\phi)(t)=&-\frac{1}{2\pi}\int_0^{2\pi}\ln\Big(4\sin^2\frac{t-\zeta}{2}\Big)\phi(\zeta)\,\mathrm{d}\zeta+\frac{\mathrm{i}}{2\pi}\int_{0}^{2\pi}\phi(\zeta)\mathrm{d}\zeta\\
	\overset{\text{def}}{=}&(\widetilde{S}_0\phi)(t)+(M_0\phi)(t),\\
	(S_1\phi)(t)=&\frac{1}{4\pi}\int_0^{2\pi}\ln\Big(4\sin^2\frac{t-\zeta}{2}\Big)\sin(t-\zeta)\phi(\zeta)\mathrm{d}\zeta,\\
	(S_2\phi)(t)=&-\frac{1}{8\pi}\int_0^{2\pi}\ln\Big(4\sin^2\frac{t-\zeta}{2}\Big)\sin^2(t-\zeta)\phi(\zeta)\mathrm{d}\zeta,
\end{align*}
which are bounded for any $p\geq0$. We split the operators $S$, $L$, $K$, $R$, and $H$ as
\begin{align*}
	&S=\widetilde{S}_0+E_1S_2+\widetilde{S}, \quad L=E_2S_2+\widetilde{L}, \quad K=-E_2S_2 +\widetilde{K},\\
	&R=E_1\widetilde{S}_0+2E_3S_1+E_4S_2+\widetilde{R}, \quad
	H=\frac{1}{2}E_1\widetilde{S}_0+E_3S_1+E_5S_2+\widetilde{H},
\end{align*}
where 
\begin{align*}
	&E_1\phi=\kappa^2|\gamma'|^2\phi,\quad E_2\phi=\kappa^2n\cdot \gamma''\phi,\quad E_3\phi=\kappa^2\gamma'\cdot \gamma''\phi,\\
	&E_4\phi=\kappa^4|\gamma'|^4\phi+2\kappa^2\gamma'\cdot \gamma'''\phi,\quad E_5\phi=-\frac{3}{4}\kappa^4|\gamma'|^4\phi+\kappa^2\gamma'\cdot \gamma'''\phi
\end{align*}
are bounded operators on $H^{p}[0,2\pi]$, and $E_1$ is isomorphic on $H^{p}[0,2\pi]$ since $|\gamma'(t)|>0$ is analytic (cf. \cite[Theorem 3.1]{DLL2021}). Here, $\widetilde{S}, \widetilde{L}, \widetilde{K}, \widetilde{R}, \widetilde{H}: H^{p}[0,2\pi]\rightarrow H^{p+4}[0,2\pi]$ are bounded, and they are integral operators  of the form 
\begin{align*}
	(Q\phi)(t)&=
	\int_0^{2\pi}\ln\Big(4\sin^2\frac{t-\zeta}{2}\Big)q_1(t,\zeta)\phi(\zeta)\mathrm{d}\zeta+\int_0^{2\pi}q_2(t,\zeta)\phi(\zeta)\mathrm{d}\zeta\\
	&\overset{\text{def}}{=}(Q_1\phi)(t)+(Q_2\phi)(t),
\end{align*}
where $q_1$ satisfies $q_1(t,t)=\partial_tq_1(t,t)=\partial^2_{tt}q_1(t,t)=0$.

Hence the integral equations \eqref{direct parafield} can be reformulated into the operator form
\begin{align}\label{Matrixequation}
	\mathcal{A}\psi:=(\mathcal{N}+\mathcal{D})\psi = \eta,
\end{align}
where $\psi=(\psi_1,\psi_2)^\top$, $\eta=(\eta_1, \eta_2)^\top$, and
\begin{align*}
	\mathcal{N}&=\left[                  
	\begin{array}{cc}
		I    & 0 \\ 
		T_0  & -I 
	\end{array}
	\right]+\left[                  
	\begin{array}{cc}
		E_2S_2        & S_0+E_1S_2 \\ 
		\frac{1}{2}E_1S_0+E_3S_1+E_6S_2  & -E_2S_2 
	\end{array}
	\right]\overset{\text{def}}{=}\mathcal{N}_1+\mathcal{N}_2,\\
	\mathcal{D}&=\left[                  
	\begin{array}{cc}
		\widetilde{L}    & \widetilde{S}-M_0 \\ 
		\widetilde{R}-\widetilde{H}-M_0-\frac{1}{2}E_1M_0  & \widetilde{K}
	\end{array}
	\right]
\end{align*}
with $E_6\phi:=E_4\phi-E_5\phi$. It can be seen that 
$\mathcal{N}: H^{p}[0,2\pi]^2\rightarrow H^{p-1}[0,2\pi]^2$ and $\mathcal{D}: H^{p}[0,2\pi]^2\rightarrow H^{p+4}[0,2\pi]^2$ are bounded operators.

\begin{theorem}\label{isomorphism}
	For all $p\geq0$, the operator $\mathcal{N}^2: H^p[0,2\pi]^2\rightarrow H^{p+2}[0,2\pi]^2$ satisfies
	\[
		\mathcal{N}^2\psi=(\mathcal{U}+\mathcal{J})\psi\quad \forall\psi\in H^p[0,2\pi]^2,
	\]
	where $\mathcal{J}: H^p[0,2\pi]^2\rightarrow H^{p+2}[0,2\pi]^2$ is a compact operator and $\mathcal{U}: H^p[0,2\pi]^2\rightarrow H^{p+2}[0,2\pi]^2$ is an isomorphism given by
	\begin{align*}
		\mathcal{U}=\left[                  
		\begin{array}{cc}
			I+S_0T_0+E_1S_0S_0   & 0 \\ 
			0        & I+T_0S_0+E_1S_0S_0 
		\end{array}
		\right].
	\end{align*}
\end{theorem}

\begin{proof} 
	Let $f_m(t):=\mathrm{e}^{\mathrm{i}mt}, m\in\mathbb{Z}$ be the trigonometric basis functions. A simple calculation yields
	\begin{align*}
		T_0f_m&=\lambda_mf_m, \quad\lambda_m=
		\begin{cases}
			-|m|, &m=\pm1,\pm2,\cdots,\\
			\mathrm{i},      &m=0,
		\end{cases}
		\\
		S_0f_m&=\xi_mf_m, \quad\xi_m=
		\begin{cases}
			\frac{1}{|m|}, &m=\pm1,\pm2,\cdots,\\
			\mathrm{i},      &m=0.
		\end{cases}
	\end{align*}
	It is clear to note that $S_0: H^{p}[0,2\pi]\rightarrow H^{p+1}[0,2\pi]$ and $T_0: H^{p}[0,2\pi]\rightarrow H^{p-1}[0,2\pi]$ are isomorphic for any $p\geq0$. Since $E_1$ is isomorphic from $H^{p}[0,2\pi]$ into $H^{p}[0,2\pi]$ and $(I+T_0S_0)f_m=(I+S_0T_0)f_m=0$, we have  
	\begin{align}\label{isomor}
		\big(E_1^{-1}(I+T_0S_0)+S_0S_0\big)f_m=	\big(E_1^{-1}(I+S_0T_0)+S_0S_0\big)f_m=\xi_m^2f_m,
	\end{align}
	which implies that $\mathcal{U}$ is isomorphic from $H^p[0,2\pi]^2$ to $H^{p+2}[0,2\pi]^2$. 
	
	A straightforward calculation yields 
	\begin{align*}
		\mathcal{N}^2
		&=\left[                  
		\begin{array}{cc}
			\frac{1}{2}S_0E_1S_0+E_1S_2T_0    & 0 \\ 
			T_0E_2S_2-E_2S_2T_0   & \frac{1}{2}E_1S_0S_0+T_0E_1S_2
		\end{array}
		\right]+\mathcal{J}_1\\
		&\quad+\left[                  
		\begin{array}{cc}
			I+S_0T_0   & 0 \\ 
			0   & I+T_0S_0 
		\end{array}
		\right]
	\end{align*}
	where 
	\begin{align*}
		\mathcal{J}_1&=\left[                  
		\begin{array}{cc}
			F_1   & F_2 \\ 
			F_3   & F_4
		\end{array}
		\right]
	\end{align*}
	with the entries being defined by 
	\begin{align}\label{F1234}
		\begin{split}
			F_1&=2E_2S_2+E_2S_2E_2S_2+S_0(E_3S_1+E_6S_2)+E_1S_2\big(\frac{1}{2}E_1S_0+E_3S_1+E_6S_2\big),\\
			F_2&=E_2S_2(S_0+E_1S_2)-(S_0+E_1S_2)E_2S_2,\\
			F_3&=\big(\frac{1}{2}E_1S_0+E_3S_1+E_6S_2\big)E_2S_2-E_2S_2\big(\frac{1}{2}E_1S_0+E_3S_1+E_6S_2\big),\\
			F_4&=2E_2S_2+E_2S_2E_2S_2+(E_3S_1+E_6S_2)S_0+\big(\frac{1}{2}E_1S_0+E_3S_1+E_6S_2\big)E_1S_2.
		\end{split}
	\end{align}
	Since $F_j, j=1,2,3,4$ are bounded from $H^p[0,2\pi]$ into $H^{p+3}[0,2\pi]$, $\mathcal{J}_1$ is bounded from $H^p[0,2\pi]$ into $H^{p+3}[0,2\pi]$, which shows that $\mathcal{J}_1$ is compact from $H^p[0,2\pi]$ into $H^{p+2}[0,2\pi]$.
	
	Using the identities
	\begin{align*}
		(T_0E_1S_2\varphi)(t)&=\frac{1}{2\pi}\int_{0}^{2\pi}\cot\frac{\zeta-t}{2}(E_1S_2\varphi)'(\zeta)\mathrm{d}\zeta+\frac{\mathrm{i}}{2\pi}\int_{0}^{2\pi}(E_1S_2\varphi)(\zeta)\mathrm{d}\zeta\\
		&=-\frac{1}{2\pi}\int_{0}^{2\pi}\ln\Big(4\sin^2\frac{\zeta-t}{2}\Big)(E_1S_2\varphi)''(\zeta)\mathrm{d}\zeta+(M_0E_1S_2\varphi)(t)\\
		&=\big(\widetilde{S}_0\{(E_1S_2\varphi)''\}\big)(t)+(M_0E_1S_2\varphi)(t)
	\end{align*}
	and 
	\begin{align*}
		(E_1S_2\varphi)''(\zeta)=&-\frac{\kappa^2}{8\pi}\int_0^{2\pi}\frac{\partial^2}{\partial{\zeta^2}}\biggl\{|\gamma'(\zeta)|^2\ln\Big(4\sin^2\frac{\zeta-s}{2}\Big)\sin^2(\zeta-s)\biggr\}\varphi(s)\mathrm{d}s\\
		=&\frac{1}{2}(E_1\widetilde{S}_0\varphi)(\zeta)+(E_7S_2\varphi)(\zeta)+4(E_3B_1\varphi)(\zeta)+(E_1B_2\varphi)(\zeta),
	\end{align*}
	where $E_7\phi=2\kappa^2(|\gamma''|^2+\gamma'\cdot \gamma''')\phi$, and 
	\begin{align*}
		(B_1\varphi)(\zeta)&=-\frac{1}{8\pi}\int_{0}^{2\pi}\ln\Big(4\sin^2\frac{\zeta-s}{2}\Big)\sin2(\zeta-s)\varphi(s)\,\mathrm{d}s-\frac{1}{8\pi}\int_{0}^{2\pi}m_1(\zeta,s)\varphi(s)\mathrm{d}s,\\
		(B_2\varphi)(\zeta)&=-\frac{1}{4\pi}\int_{0}^{2\pi}\ln\Big(4\sin^2\frac{\zeta-s}{2}\Big)\big[\cos2(\zeta-s)-1\big]\varphi(s)\,\mathrm{d}s
		\\&\quad-\frac{1}{4\pi}\int_{0}^{2\pi}m_2(\zeta,s)\varphi(s)\mathrm{d}s	
	\end{align*}
	with 
	\begin{align*}
		m_1(\zeta,s)&=\cot\frac{\zeta-s}{2}\sin^2(\zeta-s),\\
		m_2(\zeta,s)&=\cot\frac{\zeta-s}{2}\sin2(\zeta-s)-\cos^2\frac{\zeta-s}{2}
	\end{align*}
	being analytic functions, we obtain from $S_0S_0=\widetilde{S}_0\widetilde{S}_0+M_0S_0$ that
	\begin{align*}
		&\frac{1}{2}E_1S_0S_0+T_0E_1S_2\\
		&=\frac{1}{2}E_1S_0S_0+\Big(\frac{1}{2}\widetilde{S}_0E_1\widetilde{S}_0+\widetilde{S}_0E_7S_2+4\widetilde{S}_0E_3B_1+\widetilde{S}_0E_1B_2+M_0E_1S_2\Big)\\
		&=\frac{1}{2}E_1S_0S_0+\frac{1}{2}E_1\widetilde{S}_0\widetilde{S}_0+\frac{1}{2}\big(\widetilde{S}_0E_1-E_1\widetilde{S}_0\big)\widetilde{S}_0+\widetilde{S}_0E_7S_2\\
		&\qquad\qquad\qquad\qquad\qquad~~+4\widetilde{S}_0E_3B_1+\widetilde{S}_0E_1B_2+M_0E_1S_2\\
		&=E_1S_0S_0+J_1, 
	\end{align*}
	where
	\begin{align}\label{J1}
		J_1:=\frac{1}{2}\widetilde{B}_3\widetilde{S}_0+\widetilde{S}_0E_7S_2+4\widetilde{S}_0E_3B_1+\widetilde{S}_0E_1B_2-\frac{1}{2}E_1M_0S_0+M_0E_1S_2
	\end{align}
	is compact from $H^p[0,2\pi]$ into $H^{p+2}[0,2\pi]$ due to the fact that $B_2$ is bounded from $H^p[0,2\pi]$ into $H^{p+3}[0,2\pi]$, $B_1$ and
	\begin{align*}
		(\widetilde{B}_3\varphi)(t):&=(\widetilde{S}_0E_1\varphi)(t)-(E_1\widetilde{S}_0\varphi)(t)\\
		&=-\frac{\kappa^2}{2\pi}\int_{0}^{2\pi}\ln\Big(4\sin^2\frac{t-\zeta}{2}\Big)\Big(|\gamma'(\zeta)|^2-|\gamma'(t)|^2\Big)\varphi(\zeta)\,\mathrm{d}\zeta	
	\end{align*}
	are bounded from $H^p[0,2\pi]$ into $H^{p+2}[0,2\pi]$.

	In addition, using the integration by parts yields 
	\begin{align*}
		&(S_2\widetilde{T}_0\varphi)(t)\\
		&=-\frac{1}{8\pi}\int_{0}^{2\pi}\ln\Big(4\sin^2\frac{t-\zeta}{2}\Big)\sin^2(t-\zeta)(\widetilde{T}_0\varphi)(\zeta)\,\mathrm{d}\zeta\\
		&=\frac{1}{16\pi^2}\int_{0}^{2\pi}\ln\Big(4\sin^2\frac{\zeta-t}{2}\Big)\sin^2(t-\zeta)\frac{\partial}{\partial\zeta}\Biggl\{\int_{0}^{2\pi}\ln\Big(4\sin^2\frac{\zeta-s}{2}\Big)\varphi'(s)\,\mathrm{d}s\Biggr\}\,\mathrm{d}\zeta\\
		&=-\frac{1}{16\pi^2}\int_{0}^{2\pi}\Big(\int_{0}^{2\pi}\ln\Big(4\sin^2\frac{\zeta-s}{2}\Big)\varphi'(s)\,\mathrm{d}s\Big)\frac{\partial}{\partial\zeta}\\
		&\qquad\qquad\qquad\qquad\qquad\qquad\qquad\qquad\qquad\Bigl\{\ln\Big(4\sin^2\frac{\zeta-t}{2}\Big)\sin^2(t-\zeta)\Bigr\}\,\mathrm{d}\zeta\\
		&=\frac{1}{16\pi^2}\int_{0}^{2\pi}\Big(\int_{0}^{2\pi}\cot\frac{s-\zeta}{2}\varphi(s)\,\mathrm{d}s\Big)\Bigl\{\cot\frac{\zeta-t}{2}\sin^2(\zeta-t)\\
		&\qquad\qquad\qquad\qquad\qquad\qquad\qquad\qquad\qquad+\ln\Big(4\sin^2\frac{\zeta-t}{2}\Big)\sin2(\zeta-t)\Bigr\}\,\mathrm{d}\zeta\\
		&=\frac{1}{16\pi^2}\int_{0}^{2\pi}\Bigl\{2m_2(\zeta,t)+2\ln\Big(4\sin^2\frac{\zeta-t}{2}\Big)\cos2(\zeta-t)\Bigr\}\\
		&\qquad\qquad\qquad\qquad\qquad\qquad\qquad\qquad\qquad\Big(\int_{0}^{2\pi}\ln\Big(4\sin^2\frac{\zeta-s}{2}\Big)\varphi(s)\,\mathrm{d}s\Big)\,\mathrm{d}\zeta\\
		&=\frac{1}{2}\widetilde{S}_0\widetilde{S}_0\varphi+B_2\widetilde{S}_0\varphi.
	\end{align*}
	Combining $E_1S_2T_0=E_1S_2\widetilde{T}_0+E_1S_2M_0$ and $S_0E_1S_0=E_1S_0S_0+B_3S_0$, where
	\begin{align*}
		(B_3\varphi)(t):&=-\frac{\kappa^2}{2\pi}\int_{0}^{2\pi}\ln\Big(4\sin^2\frac{t-\zeta}{2}\Big)\Big(|\gamma'(\zeta)|^2-|\gamma'(t)|^2\Big)\varphi(\zeta)\,\mathrm{d}\zeta\\
		&\quad +\frac{\mathrm{i}\kappa^2}{2\pi}\int_{0}^{2\pi}
		\Big(|\gamma'(\zeta)|^2-|\gamma'(t)|^2\Big)\varphi(\zeta)\,\mathrm{d}\zeta,
	\end{align*}
	we obtain 
	\begin{align*}
		\frac{1}{2}S_0E_1S_0+E_1S_2T_0 =&\Big(\frac{1}{2}E_1S_0S_0+\frac{1}{2}B_3S_0\Big)+\frac{1}{2}E_1\widetilde{S}_0\widetilde{S}_0+E_1B_2\widetilde{S}_0+E_1S_2M_0\\
		=&E_1S_0S_0+J_2,
	\end{align*}
	where $J_2:=\frac{1}{2}B_3S_0+E_1B_2\widetilde{S}_0-\frac{1}{2}E_1M_0S_0+E_1S_2M_0$ is compact from $H^p[0,2\pi]$ into $H^{p+2}[0,2\pi]$, since $B_2: H^p[0,2\pi]\rightarrow H^{p+3}[0,2\pi]$ and $B_3: H^p[0,2\pi]\rightarrow H^{p+2}[0,2\pi]$ are bounded.
	
	Analogously, we observe that
	\begin{align*}
		&T_0E_2S_2-E_2S_2T_0\\
		&=\Big(\frac{1}{2}\widetilde{S}_0E_2\widetilde{S}_0+\widetilde{S}_0E_8S_2+2\widetilde{S}_0E_9B_1+\widetilde{S}_0E_2B_2+M_0E_2S_2\Big)\\
		&\qquad\qquad\qquad\qquad\qquad\qquad\qquad\quad~~-\frac{1}{2}E_2\widetilde{S}_0\widetilde{S}_0-E_2B_2\widetilde{S}_0-E_2S_2M_0\\
		&=\frac{1}{2}(\widetilde{S}_0E_2-E_2\widetilde{S}_0)\widetilde{S}_0+\widetilde{S}_0E_8S_2+2\widetilde{S}_0E_9B_1+\widetilde{S}_0E_2B_2-E_2B_2\widetilde{S}_0:=J_3
	\end{align*}
	is compact from $H^p[0,2\pi]$ into $H^{p+2}[0,2\pi]$, where $E_8\phi=\kappa^2n\cdot \gamma^{(4)}\phi$ and $E_9\phi=\kappa^2n\cdot \gamma'''\phi$ are bounded from $H^p[0,2\pi]$ into $H^{p}[0,2\pi]$, and $\widetilde{B}_4:=\widetilde{S}_0E_2-E_2\widetilde{S}_0$ is bounded from $H^p[0,2\pi]$ into $H^{p+2}[0,2\pi]$.
	
	The proof is completed by defining the operator
	\begin{align}\label{Jdef}
		\mathcal{J}=\mathcal{J}_1+\mathcal{J}_2,
	\end{align}
	where
	\begin{align*}
		\mathcal{J}_1=\left[                  
		\begin{array}{cc}
			F_1   & F_2 \\ 
			F_3   & F_4
		\end{array}
		\right],\quad
		\mathcal{J}_2=\left[                  
		\begin{array}{cc}
			J_2  & 0 \\ 
			J_3  & J_1
		\end{array}
		\right].
	\end{align*}
\end{proof}

The well-posedness of the boundary integral equation \eqref{Matrixequation} is stated as follows. 

\begin{theorem}\label{exiuni}
For any $p\geq0$, the operator equation \eqref{Matrixequation} admits a unique solution in $H^p[0,2\pi]^2$. 
\end{theorem}

\begin{proof}
	It follows from Theorem \ref{unique2} that the operator $\mathcal{A}=\mathcal{N}+\mathcal{D}$ in \eqref{Matrixequation} is injective. Hence $\mathcal{A}^2$ is injective. Consider the operator equation
	\begin{align}\label{Matrixeqn}
		\mathcal{A}^2\psi=(\mathcal{N}+\mathcal{D})^2\psi=\mathcal{A}\eta. 
	\end{align}
	Noting the boundedness of $\mathcal{N}: H^{p}[0,2\pi]^2\rightarrow H^{p-1}[0,2\pi]^2$ and $\mathcal{D}: H^{p}[0,2\pi]^2\rightarrow H^{p+4}[0,2\pi]^2$, we have from Theorem \ref{isomorphism} that 
	\begin{align}\label{exiuni-2}
		(\mathcal{N}+\mathcal{D})^2=\mathcal{U}+\mathcal{J}+\mathcal{N}\mathcal{D}+\mathcal{D}\mathcal{N}+\mathcal{D}^2,
	\end{align}
	where $\mathcal{U}$ is isomorphic and $\mathcal{J}+\mathcal{N}\mathcal{D}+\mathcal{D}\mathcal{N}+\mathcal{D}^2$ is compact from $H^p[0,2\pi]^2$ into $H^{p+2}[0,2\pi]^2$ for $p\geq0$. It follows from the Fredholm alternative that 
	the operator equation \eqref{Matrixeqn} has a unique solution $\psi\in H^{p}[0,2\pi]^2$, which shows that $\psi\in H^{p}[0,2\pi]^2$ is also the solution to $\mathcal{A}\psi=\eta$ due to $\mathcal{A}(\mathcal{A}\psi-\eta)=0$ and the injectivity of $\mathcal{A}$. 
\end{proof}

\section{Collocation method}\label{s_cm}

Combining \eqref{Matrixeqn} and \eqref{exiuni-2} yields 
\[
\mathcal{U}\psi+(\mathcal{J}+\mathcal{V})\psi=\mathcal{A}\eta,
\]
where $\mathcal{V}=\mathcal{N}\mathcal{D}+\mathcal{D}\mathcal{N}+\mathcal{D}^2$
is compact from $H^p[0,2\pi]^2$ into $H^{p+2}[0,2\pi]^2$. Equivalently, we consider the operator equation
\begin{align}\label{Matrixequation2}
	\mathcal{E}^{-1}\mathcal{U}\psi+\mathcal{E}^{-1}(\mathcal{J}+\mathcal{V})\psi=\mathcal{E}^{-1}\mathcal{A}\eta,
\end{align}
where 
\begin{align*}
	\mathcal{E}=\left[                  
	\begin{array}{cc}
		E_1  & 0 \\ 
		0        & E_1
	\end{array}
	\right], \quad \mathcal{E}^{-1}\mathcal{U}=\left[                  
	\begin{array}{cc}
		E_1^{-1}(I+S_0T_0)+S_0S_0  & 0 \\ 
		0        & E_1^{-1}(I+T_0S_0)+S_0S_0
	\end{array}
	\right].
\end{align*}
It is clear to note that $\mathcal{E}^{-1}\mathcal{U}$ is isomorphic and $\mathcal{E}^{-1}(\mathcal{J}+\mathcal{V})$ is compact from $H^p[0,2\pi]^2$ into $H^{p+2}[0,2\pi]^2$ for $p\geq0$. 

In this section, we adopt the collocation method to study the convergence of the semi- and full-discretization of the boundary integral equation \eqref{Matrixequation2}. A related work can be found in \cite{DLL2021} on the convergence analysis of the collocation method for solving the elastic obstacle scattering problem.

\subsection{Semi-discrete scheme}

Denote by $X_n$ the space of $n$-th order trigonometric polynomials. For any $\varphi\in X_n$, it has the form 
\[
	\varphi(t)=\sum_{m=0}^n a_m\cos mt+\sum_{m=1}^{n-1}b_m\sin mt.
\]
Let $P_n: H^p[0,2\pi]\rightarrow X_n$ be the interpolation operator, i.e., $(P_ng)(\zeta_j^{(n)})=g(\zeta_j^{(n)})$, where the interpolation points $\zeta_j^{(n)}:=\pi j/n$,
$j=0,\cdots,2n-1$. Clearly, $P_n$ is bounded. 

Let $X_n^2=\{\psi=(\psi_1,\psi_2)^\top: \psi_1\in X_n,
\psi_2\in X_n\}$ and define the interpolation operator $\mathcal{P}_n:
H^p[0,2\pi]^2\rightarrow X_n^2$ by
$$
\mathcal{P}_ng=(P_ng_1,P_ng_2)^\top\quad\forall\,g=(g_1,g_2)\in H^p[0,2\pi]^2.
$$
It is clear to note that $X_n^2$ is unisolvent with respect to the points
$\{\zeta_j^{(n)}\}_{j=0}^{2n-1}$. By \cite[Theorem
11.8]{Kress2014-book}, we have 
\begin{align}\label{interpolationerror}
	\|\mathcal{P}_ng-g\|_{q}\leq\frac{C}{n^{p-q}}\|g\|_{p}\quad\forall\,g\in H^p[0,2\pi]^2,
\end{align}
where $0\leq q\leq p, \frac{1}{2}<p$, and $C$ is a positive constant depending on $p$ and $q$.

Hence $\psi=(\psi_1,\psi_2)^\top$ can be approximated by  $\psi^n=(\psi^n_1,\psi^n_2)^\top\in X_n^2$,
which satisfies the approximate operator equation
\begin{align}\label{semicollocation}
	\mathcal{E}^{-1}\mathcal{U}\psi^n +
	\mathcal{P}_n[\mathcal{E}^{-1}(\mathcal{J}+\mathcal{V})]\psi^n =
	\mathcal{P}_n(\mathcal{E}^{-1}\mathcal{A})\eta,
\end{align}
where $\psi^n$
satisfies $\mathcal{P}_n(\mathcal{E}^{-1}\mathcal{U})\psi^n=\mathcal{E}^{-1}\mathcal{U}\psi^n$ by noting \eqref{isomor}. 

The unique solvability of \eqref{semicollocation} and the error estimate of the solution are given in the following theorem. Based on the fact that $\mathcal{J}, \mathcal{V}: H^p[0,2\pi]^2\rightarrow
H^{p+3}[0,2\pi]^2$ are bounded for any $p\geq 0$, the proof is similar to that of \cite[Theorem 4.1]{DLL2021} and is omitted for brevity.

\begin{theorem}\label{semi-convergence}
Let $\psi$ be the unique solution to \eqref{Matrixequation}. For sufficiently large $n$, the approximate operator equation \eqref{semicollocation} has a unique solution $\psi^n$, which satisfies 
	\[
		\|\psi^n-\psi\|_p\leq
		C\|\mathcal{P}_n(\mathcal{E}^{-1}\mathcal{U})\psi-\mathcal{E}^{-1}\mathcal{U}\psi\|_{p+2},
	\]
	where $C$ is a positive constant depending on 
	$\mathcal{E}^{-1}\mathcal{J}$, $\mathcal{E}^{-1}\mathcal{V}$ and
	$\mathcal{E}^{-1}\mathcal{U}$.
\end{theorem}

\subsection{Full-discrete scheme}

Based on the Lagrange basis (cf. \cite[eqn. $(11.12)$]{Kress2014-book})
\[
\mathfrak{L}_j(t)=\frac{1}{2n}\Big\{1+2\sum_{k=1}^{n-1}\cos k(t-\zeta_j^{(n)})+\cos n(t-\zeta_j^{(n)})\Big\}, \quad j=0,1,\cdots,2n-1,
\]
an approximate solution
$\widetilde\psi^n\in X_n^2$ of \eqref{semicollocation} can be written as 
\[
\widetilde\psi^n(t)=\big(\widetilde\psi^n_1(t),
\widetilde\psi^n_2(t)\big)^\top=\bigg(\sum_{j=0}^{2n-1}
\widetilde\psi^n_1(\zeta_j^{(n)})\mathfrak{L}_j(t) ,
\sum_{j=0}^{2n-1}\widetilde\psi^n_2(\zeta_j^{(n)})
\mathfrak{L}_j(t)\bigg)^\top. 
\] 
Moreover, it satisfies
\begin{align}\label{fullcollocation}
	\mathcal{E}_n^{-1}\mathcal{U}_n\widetilde\psi^n+\mathcal{P}_n[\mathcal{E}_n^{-1}(\mathcal{J}_n+\mathcal{V}_n)]\widetilde\psi^n=\mathcal{P}_n(\mathcal{E}_n^{-1}\mathcal{A}_n)\eta, 
\end{align}
where $\mathcal{A}_n=\mathcal{N}_n+\mathcal{D}_n=\mathcal{N}_{1,n}+\mathcal{N}_{2,n}+\mathcal{D}_{n}$, $\mathcal{J}_n=\mathcal{J}_{1,n}+\mathcal{J}_{2,n}$, $\mathcal{V}_n=\mathcal{N}_n\mathcal{D}_n+\mathcal{D}_n\mathcal{N}_n+\mathcal{D}_n\mathcal{D}_n$. The involved quadrature operators are given by
\begin{align}\label{TSE}
	T_{0,n}=T_0P_n, \quad S_{j,n}=S_jP_n, \quad E_{i,n}=E_i
\end{align}
for $j=0,1,2$, $i=1,\cdots,9$, and the quadrature operators corresponding the integral operators of the form
\begin{align*}
	(B\phi)(t)&=
	\int_0^{2\pi}\ln\Big(4\sin^2\frac{t-\zeta}{2}\Big)b_1(t,\zeta)\phi(\zeta)\,\mathrm{d}\zeta+\int_0^{2\pi}b_2(t,\zeta)\phi(\zeta)\,\mathrm{d}\zeta\\
	&\overset{\text{def}}{=}(B_{1}\phi)(t)+(B_{2}\phi)(t)
\end{align*}
are defined by
\begin{align*}
	(B_n\phi)(t)&=
	\int_0^{2\pi}\ln\Big(4\sin^2\frac{t-\zeta}{2}\Big)P_n[b_1(t,\zeta)\phi(\zeta)]\,\mathrm{d}\zeta+\int_0^{2\pi}P_n[b_2(t,\zeta)\phi(\zeta)]\,\mathrm{d}\zeta\\
	&\overset{\text{def}}{=}(B_{1,n}\phi)(t)+(B_{2,n}\phi)(t).
\end{align*}
Then we have $\mathcal{N}_{1,n}=\mathcal{N}_{1}\mathcal{P}_n$, $\mathcal{N}_{2,n}=\mathcal{N}_{2}\mathcal{P}_n$, and
\begin{align*}
	\mathcal{E}_n&=\left[             \begin{array}{cc}
		E_{1,n} & 0 \\ 
		0        & E_{1,n}
	\end{array}
	\right],\quad
	\mathcal{J}_{1,n}=\left[                  
	\begin{array}{cc}
		F_{1,n}   & F_{2,n} \\ 
		F_{3,n}   & F_{4,n}
	\end{array}
	\right], \quad
	\mathcal{J}_{2,n}=\left[                  
	\begin{array}{cc}
		J_{2,n}  & 0 \\ 
		J_{3,n}   & J_{1,n}
	\end{array}
	\right],
\end{align*}
where $F_{j,n}$ and $J_{i,n}$ are quadrature operators corresponding to the integral operators $F_j$ and $J_i$, respectively. It is clear to note that $\mathcal{E}_n^{-1}\mathcal{U}_n\widetilde\psi^n=\mathcal{E}^{-1}\mathcal{U}\widetilde\psi^n$ and $\mathcal{N}_n\widetilde\psi^n=\mathcal{N}\widetilde\psi^n$ for any $\widetilde\psi^n\in X_n^2$. 

Our goal is to examine the convergence of the fully discretized equation \eqref{fullcollocation}. The function $\mathcal{D}\psi$ can be split into the following two parts:
\[
	(\mathcal{D}\psi)(t)=\int_0^{2\pi}\ln\Big(4\sin^2\frac{t-\zeta}{2}
	\Big)\alpha(t,\zeta)\psi(\zeta)\,\mathrm{d}\zeta+\int_0^{2\pi}\beta(t,
	\zeta)\psi(\zeta)\,\mathrm{d}\zeta,
\]
where 
\[
\alpha(t,\zeta)=\left[\begin{array}{cc}
	\alpha_1(t,\zeta)& \alpha_2(t,\zeta)\\ 
	\alpha_3(t,\zeta)& \alpha_4(t,\zeta)
\end{array}
\right], \quad \beta(t,\zeta)=\left[\begin{array}{cc}
	\beta_1(t,\zeta)& \beta_2(t,\zeta)\\ 
	\beta_3(t,\zeta)& \beta_4(t,\zeta)
\end{array}
\right]
\]
with $\alpha_j(t,t)=\partial_t \alpha_j(t,t)=\partial^2_{tt} \alpha_j(t,t)=0, j=1,2,3,4$ and $\alpha_j, \beta_j$ being analytic.

\begin{theorem}\label{Bn-estimate}
	Let $0\leq q\leq p$ and $p>1/2$. Then for any $\psi\in X_n^2$ and $\chi\in H^p[0,2\pi]^2$, it holds that 
	\begin{align}
		\label{Bestimate}
		\|\mathcal{D}_n\psi-\mathcal{D}\psi\|_{q+3}\leq C\frac{1}{n^{p+1-q}}\|\psi\|_{p},\quad \|\mathcal{D}_n\chi-\mathcal{D}\chi\|_{q+3}\leq \widetilde{C}\frac{1}{n^{p-q}}\|\chi\|_{p},
	\end{align}
	where $C$ and $\widetilde{C}$ are positive constants depending on $p$ and $q$.
\end{theorem}

\begin{proof}
	We write the derivative $\frac{\rm d^2}{{\rm d}t^2}(\mathcal{D}\psi)$ in form of
	\begin{align*}
		(\mathcal{D}''\psi)(t):=\frac{\rm d^2}{{\rm
				d}t^2}(\mathcal{D}\psi)(t)=\int_0^{2\pi}\ln\Big(4\sin^2\frac{t-\zeta}{2}
		\Big)\widetilde \alpha(t,\zeta)
		\psi(\zeta)\,\mathrm{d}\zeta+\int_0^{2\pi}\widetilde \beta(t,\zeta)
		\psi(\zeta)\,\mathrm{d}\zeta,
	\end{align*}
	where
	\begin{align*}
		\widetilde \alpha(t,\zeta)=\partial^2_{tt}\alpha(t,\zeta),\quad \widetilde \beta(t,\zeta) =2\cot\frac{t-\zeta}{2}\partial_t\alpha(t,\zeta)-\alpha(t,\zeta)\frac{1}{2\sin^2\frac{t-\zeta}{2}}+\partial^2_{tt}\beta(t,\zeta).
	\end{align*}
	By the interpolatory quadrature, the full discretization of $\mathcal{D}''$ can be written as 
	\begin{align*}
		(\mathcal{D}''_n\psi)(t)=\int_0^{2\pi}\ln\Big(4\sin^2\frac{t-\zeta}{2}\Big)\mathcal{P}_n\Big\{\widetilde \alpha(t,\cdot) \psi\Big\}(\zeta)\,\mathrm{d}\zeta+\int_0^{2\pi}\mathcal{P}_n\Big\{\widetilde \beta(t,\cdot) \psi\Big\}(\zeta)\,\mathrm{d}\zeta.
	\end{align*}

	Noting $p>1/2, 0\leq q\leq p, \widetilde \alpha(t,t)=0$, and the analyticity of the elements in $\widetilde \alpha(t,\zeta)$ and $\widetilde \beta(t,\zeta)$, we have from \cite[Lemma 13.21 and Theorem 12.18]{Kress2014-book} that 
	$$
	\|\mathcal{D}''_n\psi-\mathcal{D}''\psi\|_{q+1}\leq C_1\frac{1}{n^{p+1-q}}\|\psi\|_{p},\quad \|\mathcal{D}''_n\chi-\mathcal{D}''\chi\|_{q+1}\leq \widetilde{C}_1\frac{1}{n^{p-q}}\|\chi\|_{p}
	$$
	for any $\psi\in X_n^2$ and $\chi\in H^p[0,2\pi]^2$, where the positive constants $C_1$ and $\widetilde{C}_1$ depend on $p$ and $q$. Since $\mathcal{D}''_n\psi=\frac{\rm d^2}{{\rm d}t^2}(\mathcal{D}_n\psi)$, the above inequalities reduce to  
	\begin{align*}
		\|\mathcal{D}_n\psi-\mathcal{D}\psi\|_{q+3}\leq C_2\frac{1}{n^{p+1-q}}\|\psi\|_{p}, \quad \|\mathcal{D}_n\chi-\mathcal{D}\chi\|_{q+3}\leq \widetilde{C}_2\frac{1}{n^{p-q}}\|\chi\|_{p},
	\end{align*}
	where $C_2$ and $\widetilde{C}_2$ are positive constants depending on $p$ and $q$.
	By \cite[Theorem 8.13]{Kress2014-book}, the proof is completed since the above inequalities hold for any $q$ satisfying $0\leq q\leq p$ and $p>1/2$.
\end{proof}

\begin{remark}\label{estimate}
	For any $\varphi\in X_n$ and $\tilde{\chi}\in H^p[0,2\pi]$, we may follow the proofs of \cite[Theorem 4.2]{DLL2021} and \cite[Lemma 13.21]{Kress2014-book} to show that $S_2, B_2$ satisfy the following type of estimates:
	\[
		\|S_{2,n}\varphi-S_{2}\varphi\|_{q+2}
		\leq C_1\frac{1}{n^{p+1-q}}\|\varphi\|_{p},\quad  \|S_{2,n}\tilde{\chi}-S_{2}\tilde{\chi}\|_{q+2}\leq \widetilde{C}_1\frac{1}{n^{p-q}}\|\tilde{\chi}\|_{p}, 
	\]
	and $S_1, B_1, B_3, \widetilde{B}_3, \widetilde{B}_4$ satisfy the following type of estimates: 
	\[
		\|S_{1,n}\varphi-S_{1}\varphi\|_{q+1}
		\leq C_2\frac{1}{n^{p+1-q}}\|\varphi\|_{p},\quad \|S_{1,n}\tilde{\chi}-S_{1}\tilde{\chi}\|_{q+1}\leq \widetilde{C}_2\frac{1}{n^{p-q}}\|\tilde{\chi}\|_{p},
	\]
	where $0\leq q\leq p$, $p>\frac{1}{2}$, and the positive constants $C_j, \widetilde{C}_j, j=1, 2$ depend on $p$ and $q$.
\end{remark}

Hereafter, the notation $a\preceq b$ stands for $a\leq Cb$, where $C > 0$ is a constant depending on $p$ and may change step by step in the proofs.

\begin{theorem}\label{Kn-estimate}
	Let $p>\frac{3}{2}$. Then the following estimate holds: 
	$$
	\|\mathcal{P}_n[\mathcal{E}_n^{-1}\mathcal{V}_n-\mathcal{E}^{-1}\mathcal{V}]\psi\|_{p+2}\preceq \frac{1}{n}\|\psi\|_{p}\quad\forall\,\psi\in X_n^2. 
	$$
\end{theorem}

\begin{proof}
	The proof is motivated by \cite[Theorem 4.3]{DLL2021}. Recalling the boundedness of $\mathcal{N}: H^{p}[0,2\pi]^2\rightarrow H^{p-1}[0,2\pi]^2$ and using \eqref{interpolationerror}, we have for any $\chi\in H^p[0,2\pi]^2, 0\leq q\leq p, p>1/2$ that 
	\begin{align*}
		\|(\mathcal{N}_n-\mathcal{N})\chi\|_{q-1}=\|\mathcal{N}(\mathcal{P}_n\chi-\chi)\|_{q-1}\leq C_1\|(\mathcal{P}_n\chi-\chi)\|_{q}\leq\frac{C_2}{n^{p-q}}\|\chi\|_{p}, 
	\end{align*} 
	where $C_1$ and $C_2$ are positive constants depending on $q$ and $p, q$, respectively. For any $p>1/2$, it is clear to note that $\mathcal{N}_n$ and $\mathcal{N}_n-\mathcal{N}$ are uniformly bounded from $H^p[0,2\pi]^2$ to $H^{p-1}[0,2\pi]^2$.
	
	For any $\psi\in X_n^2$, it follows from Theorem \ref{Bn-estimate} and $\mathcal{N}_n\psi=\mathcal{N}\psi$ that 
	\begin{align*}
		&\|\mathcal{N}_n\mathcal{D}_n\psi-\mathcal{N}\mathcal{D}\psi\|_{p+2}\\
		&\leq
		\|\mathcal{N}_n(\mathcal{D}_n-\mathcal{D})\psi\|_{p+2}+\|(\mathcal{N}
		_n-\mathcal{N})(\mathcal{D}\psi-\mathcal{P}_n\mathcal{D}\psi)\|_{p+2}
		+\|(\mathcal{N}_n-\mathcal{N})\mathcal{P}_n\mathcal{D}\psi\|_{p+2}\\
		&\preceq \|(\mathcal{D}_n-\mathcal{D})\psi\|_{p+3}+\|\mathcal{D}
		\psi-\mathcal{P} _n\mathcal{D}\psi\|_{p+3}\\
		&\preceq 1/n\|\psi\|_{p}+1/n\|\mathcal{D}\psi\|_{p+4}
		\preceq1/n\|\psi\|_{p}.
	\end{align*}
	Moreover, by Theorem \ref{Bn-estimate}, we have for any $p>1/2$ that 
	$\mathcal{D}_n$ and $\mathcal{D}_n-\mathcal{D}$ are uniformly bounded from
	$H^p[0,2\pi]^2$ to $H^{p+3}[0,2\pi]^2$. Combining 
	\eqref{interpolationerror} and \eqref{Bestimate}, and noting the boundedness of $\mathcal{D}: H^p[0,2\pi]^2\rightarrow H^{p+4}[0,2\pi]^2$ and the uniform boundedness of
	$\mathcal{P}_n: H^{p+2}[0,2\pi]^2\rightarrow H^{p+2}[0,2\pi]^2$, we obtain 
	\begin{align*}
		&\|\mathcal{D}^2_n\psi-\mathcal{D}^2\psi\|_{p+2}
		\leq \|\mathcal{D}^2_n\psi-\mathcal{D}^2\psi\|_{p+6}\\
		&\leq \|\mathcal{D}_n(\mathcal{D}_n-\mathcal{D})\psi\|_{p+6}+\|(\mathcal{D}_n-\mathcal{D})(\mathcal{D}\psi-\mathcal{P}_n\mathcal{D}\psi)\|_{p+6}+\|(\mathcal{D}_n-\mathcal{D})\mathcal{P}_n\mathcal{D}\psi\|_{p+6}\\
		&\preceq   \|(\mathcal{D}_n-\mathcal{D})\psi\|_{p+3}+\|(\mathcal{D}\psi-\mathcal{P}_n\mathcal{D}\psi)\|_{p+3}+1/n\|\mathcal{P}_n\mathcal{D}\psi\|_{p+3}\\
		&\preceq   1/n\|\psi\|_{p}+1/n\|\mathcal{D}\psi\|_{p+4}+1/n\|\mathcal{D}\psi\|_{p+3}
		\preceq  1/n\|\psi\|_{p}. 
	\end{align*}
	With the help of $\mathcal{N}_1\psi\in X_n^2$, the boundedness of $\mathcal{N}_2: H^{p}[0,2\pi]^2\rightarrow H^{p+1}[0,2\pi]^2$ and the uniform boundedness of $\mathcal{D}_n-\mathcal{D}: H^{p-1}[0,2\pi]^2$ to $H^{p+2}[0,2\pi]^2$ and $\mathcal{P}_n: H^{p-1}[0,2\pi]^2\rightarrow H^{p-1}[0,2\pi]^2$ for $p>\frac{3}{2}$, we deduce 
	\begin{align*}
		&\|\mathcal{D}_n\mathcal{N}_n\psi-\mathcal{D}\mathcal{N}\psi\|_{p+2}=\|(\mathcal{D}_n-\mathcal{D})\mathcal{N}\psi\|_{p+2}\\
		&\leq  \|(\mathcal{D}_n-\mathcal{D})(\mathcal{N}\psi-\mathcal{P}_n\mathcal{N}\psi)\|_{p+2}+\|(\mathcal{D}_n-\mathcal{D})\mathcal{P}_n\mathcal{N}\psi\|_{p+2}\\
		&\preceq \|\mathcal{N}\psi-\mathcal{P}_n\mathcal{N}\psi\|_{p-1}+1/n\|\mathcal{P}_n\mathcal{N}\psi\|_{p-1}\\
		&\leq  \|\mathcal{N}_1\psi-\mathcal{P}_n\mathcal{N}_1\psi\|_{p-1}+\|\mathcal{N}_2\psi-\mathcal{P}_n\mathcal{N}_2\psi\|_{p-1}+1/n\|\mathcal{P}_n\mathcal{N}\psi\|_{p-1}\\
		&\preceq 
		1/n^2\|\mathcal{N}_2\psi\|_{p+1}+1/n\|\mathcal{N}\psi\|_{p-1}
		\preceq 1/n\|\psi\|_{p}.
	\end{align*}
	Combining the above estimates yields 
	\begin{align*}
		\|\mathcal{V}_n\psi-\mathcal{V}\psi\|_{p+2}
		&\leq 
		\|\mathcal{N}_n\mathcal{D}_n\psi-\mathcal{N}\mathcal{D}\psi\|_{p+2}+\|\mathcal{D}_n\mathcal{N}_n\psi-\mathcal{D}\mathcal{N}\psi\|_{p+2}\\
		&\quad+\|\mathcal{D}_n^2\psi-\mathcal{D}^2\psi\|_{p+2} \\
		&\preceq 1/n\|\psi\|_{p},
	\end{align*}
	which completes the proof by noting that the operators $\mathcal{E}^{-1},\mathcal{P}_n: H^{p+2}[0,2\pi]^2\rightarrow
	H^{p+2}[0,2\pi]^2$ are uniformly bounded.  
\end{proof}

\begin{theorem}\label{Jn-estimate}
	Let $p>1/2$. Then the following estimate holds: 
	$$
	\|\mathcal{P}_n[\mathcal{E}_n^{-1}\mathcal{J}_n-\mathcal{E}^{-1}\mathcal{J}]\psi\|_{p+2}
	\preceq  \frac{1}{n}\|\psi\|_{p}\quad\forall\,\psi\in X_n^2. 
	$$
\end{theorem}

\begin{proof}
	We claim for any $\varphi\in X_n$ that 
	\begin{align} \label{FJestimate}
		\|F_{j,n}\varphi-F_j\varphi\|_{p+2}
		\preceq 
		\frac{1}{n}\|\varphi\|_{p}, \quad
		\|J_{i,n}\varphi-J_i\varphi\|_{p+2}
		\preceq
		\frac{1}{n}\|\varphi\|_{p}.
	\end{align}
	for $j=1,2,3,4$, $i=1,2,3$. In the following, we only show the proof for $j=1$, $i=1$ since the other cases can be proved similarly. 
	
	First, we have from \eqref{F1234} that 
	\begin{align}\label{F1}
		\|F_{1,n}\varphi-F_1\varphi\|_{p+2}\leq I_1+I_2+I_3+I_4 \quad\forall\, \varphi\in X_n, 
	\end{align} 
	where
	\begin{align*}
		I_1&=\|2E_{2,n}S_{2,n}\varphi-2E_2S_2\varphi\|_{p+2},\\
		I_2&=\|E_{2,n}S_{2,n}E_{2,n}S_{2,n}\varphi-E_2S_2E_2S_2\varphi\|_{p+2},\\
		I_3&=\|S_{0,n}(E_{3,n}S_{1,n}+E_{6,n}S_{2,n})\varphi-S_{0}(E_{3}S_{1}+E_{6}S_{2})\varphi\|_{p+2},\\
		I_4&=\|E_{1,n}S_{2,n}(\frac{1}{2}E_{1,n}S_{0,n}+E_{3,n}S_{1,n}+E_{6,n}S_{2,n})\varphi\\
		&\quad-E_{1}S_{2}(\frac{1}{2}E_{1}S_{0}+E_{3}S_{1}+E_{6}S_{2})\varphi\|_{p+2}.
	\end{align*}
	It is clear to note from \eqref{TSE} that $I_1=0$. Using \eqref{interpolationerror} and Remark \ref{estimate}, and following the same techniques as those in \cite[Theorem 4.4]{DLL2021}, we have
	\begin{align*}
		I_2 &\preceq\|S_{2,n}E_{2,n}S_{2,n}\varphi-S_2E_2S_2\varphi\|_{p+2}=\|(S_{2,n}-S_2)(E_2S_2\varphi-P_nE_2S_2\varphi)\|_{p+2}\\
		&\preceq\|E_2S_2\varphi-P_nE_2S_2\varphi\|_{p}\preceq\frac{1}{n^3}\|E_2S_2\varphi\|_{p+3}
		\preceq\frac{1}{n}\|\varphi\|_p. 
	\end{align*}
	By the uniform boundedness of $S_{0,n}-S_0: H^{p}[0,2\pi]^2\rightarrow H^{p+1}[0,2\pi]^2$, we get
	\begin{align*}
		I_3 &\preceq\|S_{0,n}E_{3,n}S_{1,n}\varphi-S_0E_3S_1\varphi\|_{p+2}+\|S_{0,n}E_{6,n}S_{2,n}\varphi-S_0E_6S_2\varphi\|_{p+2}\\
		&=\|(S_{0,n}-S_0)(E_3S_1\varphi-P_nE_3S_1\varphi)\|_{p+2}+\|(S_{0,n}-S_0)(E_6S_2\varphi-P_nE_6S_2\varphi)\|_{p+2}\\
		&\preceq\|E_3S_1\varphi-P_nE_3S_1\varphi\|_{p+1}+\|E_6S_2\varphi-P_nE_6S_2\varphi\|_{p+1}\\
		&\preceq\frac{1}{n}\|E_3S_1\varphi\|_{p+2}+\frac{1}{n^2}\|E_6S_2\varphi\|_{p+3}
		\preceq\frac{1}{n}\|\varphi\|_p. 
	\end{align*}
	Analogously, we may show $I_4\preceq\frac{1}{n}\|\varphi\|_p$. Combining $I_j\preceq\frac{1}{n}\|\varphi\|_p, ~j=1,2,3,4$ and \eqref{F1} leads to the first inequality in \eqref{FJestimate}. 
	
	Next, it follows from \eqref{J1} that 
	\begin{align*}
		\|J_{1,n}\varphi-J_1\varphi\|_{p+2}\leq Q_1+Q_2+Q_3+Q_4+Q_5+Q_6 \quad\forall\, \varphi\in X_n,
	\end{align*} 
	where
	\begin{align*}
		Q_1&=\|\frac{1}{2}\widetilde{B}_{3,n}\widetilde{S}_{0,n}\varphi-\frac{1}{2}\widetilde{B}_{3}\widetilde{S}_{0}\varphi\|_{p+2}, \qquad~~~~~~
		Q_2=\|\widetilde{S}_{0,n}E_{7,n}S_{2,n}\varphi-\widetilde{S}_{0}E_{7}S_{2}\varphi\|_{p+2},\\
		Q_3&=\|4\widetilde{S}_{0,n}E_{3,n}B_{1,n}\varphi-4\widetilde{S}_{0}E_{3}B_{1}\varphi\|_{p+2}, \quad~
		Q_4=\|\widetilde{S}_{0,n}E_{1,n}B_{2,n}\varphi-\widetilde{S}_{0}E_{1}B_{2}\varphi\|_{p+2},\\
		Q_5&=\|\frac{1}{2}E_{1,n}M_{0,n}S_{0,n}\varphi-\frac{1}{2}E_{1}M_{0}S_{0}\varphi\|_{p+2}, ~
		Q_6=\|M_{0,n}E_{1,n}S_{2,n}\varphi-M_{0}E_{1}S_{2}\varphi\|_{p+2}.
	\end{align*}
	Again, we have from Remark \ref{estimate} that 
	\begin{align*}
		Q_1\preceq\|(\widetilde{B}_{3,n}-\widetilde{B}_3)\widetilde{S}_0\varphi\|_{p+2}
		\preceq\frac{1}{n}\|\widetilde{S}_0\varphi\|_{p+1}\preceq\frac{1}{n}\|\varphi\|_p.
	\end{align*}
	The inequality $Q_2\preceq\frac{1}{n}\|\varphi\|_p$ can be obtained by the similar estimate of the second item of $I_3$. Noting  \eqref{interpolationerror} and the uniform boundedness of $\widetilde{S}_0, ~\widetilde{S}_{0,n}-\widetilde{S}_0: H^{p}[0,2\pi]\rightarrow H^{p+1}[0,2\pi]$, we deduce 
	\begin{align*}
		Q_3 &\preceq\|\widetilde{S}_{0,n}(E_{3,n}B_{1,n}-E_3B_1)\varphi\|_{p+2}+\|(\widetilde{S}_{0,n}-\widetilde{S}_0)(E_3B_1\varphi-P_nE_3B_1\varphi)\|_{p+2}\\
		&\preceq\|(E_{3,n}B_{1,n}-E_3B_1)\varphi\|_{p+1}+\|E_3B_1\varphi-P_nE_3B_1\varphi\|_{p+1}\\
		&\preceq\|(B_{1,n}-B_1)\varphi\|_{p+1}+\frac{1}{n}\|E_3B_1\varphi\|_{p+2}\\
		&\preceq\frac{1}{n}\|\varphi\|_p+\frac{1}{n}\|\varphi\|_p\preceq\frac{1}{n}\|\varphi\|_p.
	\end{align*}
	Analogously, $Q_4$ can be estimate as  
	\begin{align*}
	Q_4&\preceq\|(B_{2,n}-B_2)\varphi\|_{p+1}+\frac{1}{n}\|E_3B_2\varphi\|_{p+2}\\
	&\leq\|(B_{2,n}-B_2)\varphi\|_{p+2}+\frac{1}{n}\|E_3B_2\varphi\|_{p+3}\preceq\frac{1}{n}\|\varphi\|_p.
    \end{align*}
	Clearly, $Q_5=0$ by noting $S_{0,n}\varphi=S_0\varphi\in X_n$. It follows from \cite[Theorem A.45]{Kirsch2011-book} that $M_0$ is a bounded operator from $H^p[0,2\pi]$ into $H^r[0,2\pi]$ for every $-r\leq p\leq r$. Hence, we have
	\begin{align*}
		Q_6\preceq\|P_nE_1S_2\varphi-E_1S_2\varphi\|_{p+2}\preceq\frac{1}{n}\|E_1S_2\varphi\|_{p+3}\preceq\frac{1}{n}\|\varphi\|_p,
	\end{align*}
	which implies the second inequality of \eqref{FJestimate}. 
	
	Combining \eqref{Jn-estimate} and \eqref{Jdef}, we complete the proof by noting that the operators $\mathcal{E}^{-1},\mathcal{P}_n: H^{p+2}[0,2\pi]^2\rightarrow H^{p+2}[0,2\pi]^2$ are uniformly bounded. 
\end{proof}

The following result concerns the convergence of the full-discrete scheme.  Based on the uniform boundedness of the operators $F_{j,n}-F_j, J_{i,n}-J_i: H^{p}[0,2\pi]\rightarrow H^{p+2}[0,2\pi]$ from Remark \ref{estimate} and \eqref{Bestimate}, the proof is similar to that of \cite[Theorem 4.5]{DLL2021} and is omitted here for brevity. 

\begin{theorem}\label{fullconvergence}
Let $\psi$ be the unique solution to \eqref{Matrixequation}. For $p>3/2$ and sufficiently large $n$, the fully discrete equation \eqref{fullcollocation} admits a unique solution $\widetilde\psi^n$, which satisfies 
	\begin{align*}
		&\|\widetilde\psi^n-\psi\|_p\\
		&\preceq\|\mathcal{P}_n(\mathcal{E}^{-1}\mathcal{U})\psi-\mathcal{E}^{-1}\mathcal{U}\psi\|_{p+2}+\|\mathcal{P}_n[\mathcal{E}_n^{-1}(\mathcal{J}_n+\mathcal{K}_n)-\mathcal{E}^{-1}(\mathcal{J}+\mathcal{K})]\psi\|_{p+2}\\
		&\quad +\|\mathcal{P}_n[\mathcal{E}_n^{-1}\mathcal{A}_n-\mathcal{E}^{-1}\mathcal {A}]
		\eta\|_{p+2}. 
	\end{align*}
\end{theorem}

\section{Numerical experiments}\label{s_ne}

In this section, we present numerical implementation and show some examples to demonstrate the superior performance of the proposed method. In particular, we introduce a simple alternative boundary integral formulation by using a combination of the single- and single-layer potentials. 

\subsection{Double-single layer potential formulation}

As discussed in section 3, the boundary integral formulation is based on a combination of the double- and single-layer potentials. Here we introduce two different approaches to solve the corresponding full-discrete boundary integral equations. 

\subsubsection{Approach 1}

The first approach is to solve directly the equivalent full-discrete equation of \eqref{direct parafield}, i.e., 
\begin{align*}
	\Bigg(\left[     
	\begin{array}{cc}
		I     & 0 \\ 
		\widetilde{T}_0   & -I
	\end{array}
	\right]+\left[     
	\begin{array}{cc}
		L      & S \\ 
		R-H   & K
	\end{array}
	\right]\Bigg)\left[                  
	\begin{array}{c}
		\psi_1 \\ 
		\psi_2
	\end{array}
	\right]=\left[                  
	\begin{array}{c}
		\eta_1 \\ 
		\eta_2
	\end{array}
	\right],
\end{align*}
where $\widetilde{T}_0=T_0- M_0$. By the decomposition \eqref{kerdecom_chi}, we apply the simple trapezoidal for the smooth integrals: 
\begin{align*} 
	\int_{0}^{2\pi}f(\zeta)\mathrm{d}\zeta\approx\frac{\pi}{n}\sum_{j=0}^{
		2n-1}f(\zeta_j^{(n)}), 
\end{align*}
while we employ the quadrature rules via the trigonometric interpolation for the weakly singular integrals (cf. \cite{DR-shu2,Kress1995}]:
\begin{align*}
	\int_{0}^{2\pi}\ln\Big(4\sin^2\frac{t-\zeta}{2}\Big)f(\zeta)\,\mathrm{d}
	\zeta&\approx\sum_{j=0}^{2n-1}R_j^{(n)}(t)f(\zeta_j^{(n)}),\\
	\frac{1}{2\pi}\int_0^{2\pi}\cot\frac{\zeta-t}{2}f'(\zeta)\,\mathrm{d}\zeta&\approx\sum_{j=0}^{2n-1}T_j^{(n)}(t)f(\zeta_j^{(n)}),
\end{align*}
where the quadrature weights are given by
\begin{align*}
	&R_j^{(n)}(t)=-\frac{2\pi}{n}\sum_{m=1}^{n-1}\frac{1}{m}\cos\Big[m(t-\zeta_j^{(n)})\Big]
	-\frac{\pi}{n^2}\cos\Big[n(t-\zeta_j^{(n)})\Big],\\
	&T_j^{(n)}(t)=-\frac{1}{n}\sum_{m=1}^{n-1}m\cos\Big[m(t-\zeta_j^{(n)})\Big]
	-\frac{1}{2}\cos\Big[n(t-\zeta_j^{(n)})\Big].
\end{align*}

\subsubsection{Approach 2}

The second approach is to solve the
equivalent full-discrete equation of \eqref{Matrixequation}, i.e., 
\begin{align}\label{transformfull}
	\mathcal{P}_n\mathcal{N}_n\widetilde\psi^{n}+\mathcal{P}_n\mathcal{D}
	_n\widetilde\psi^{n}=\mathcal{P}_n\eta
	~\Leftrightarrow~
	A_n\widetilde\psi^{n}=\eta^{n},
\end{align}
where $A_n$ is the coefficient matrix of the full-discrete equation. As mentioned in \cite[section 5]{DLL2021}, it is more convenient to handle the equivalent full-discrete equation \eqref{transformfull} due to the simple quadrature operators $\mathcal{N}_n$ and $\mathcal{D}_n$. 

To handle the singular integrals $S_1$ and $S_2$, we use the trigonometric interpolation in the quadrature rules
\begin{align*}
	\int_0^{2\pi}\ln\Big(4\sin^2\frac{t-\zeta}{2}
	\Big)\sin(t-\zeta)f(\zeta)\,\mathrm{d}\zeta&\approx\sum_{j=0}^{2n-1}
	V_j^{(n)}(t)f(\zeta_j^{(n)}),\\
	\int_0^{2\pi}\ln\Big(4\sin^2\frac{t-\zeta}{2}
	\Big)\sin^2(t-\zeta)f(\zeta)\,\mathrm{d}\zeta&\approx\sum_{j=0}^{2n-1}
	W_j^{(n)}(t)f(\zeta_j^{(n)}),
\end{align*}
where the integral weights are defined by
\begin{align*}
	V_j^{(n)}(t)&=-\frac{\pi}{2n}\sin(\zeta_j^{(n)}-t)+\frac{2\pi}{n}\sum_{m=2}^{n-1}\frac{\sin[m(\zeta_j^{(n)}-t)]}{m^2-1}+\frac{\pi\sin[n(\zeta_j^{(n)}-t)]}{n(n^2-1)},\\
	W_j^{(n)}(t)&=\frac{\pi}{n}\Big\{\frac{1}{4}+\frac{2}{3}\cos(t-\zeta_j^{(n)})+\frac{1}{8}\cos[2(t-\zeta_j^{(n)})]\Big\}+\frac{1}{2}R_j^{(n)}(t)\\
	&\quad+\frac{\pi}{n}\sum_{m=3}^{n-1}\frac{m}{m^2-4}\cos[m(t-\zeta_j^{(n)})]+\frac{\pi}{2(n^2-4)}\cos[n(t-\zeta_j^{(n)})].
\end{align*}

\begin{remark}
	By a straightforward calculation, we get 
	\begin{align*}
		&W_j^{(n)}(t)-R_j^{(n)}(t)\sin^2(t-\zeta_j^{(n)})\\
		&=\frac{\pi}{2n^2}\sin[n(t-\zeta_j^{(n)})]\sin[2(t-\zeta_j^{(n)})]+\frac{\pi}{n^2-n}\sin [n(t-\zeta_j^{(n)})]\sin(t-\zeta_j^{(n)})\\
		&\quad -\frac{\pi}{n(n-1)(n+1)}\cos[(n-1)(t-\zeta_j^{(n)})]-\frac{\pi}{n(n-2)(n+2)}\cos[n(t-\zeta_j^{(n)})].
	\end{align*}
	It is clear to note that $W_j^{(n)}(\zeta_i^{(n)})-R_j^{(n)}(\zeta_i^{(n)})\sin^2(\zeta_i^{(n)}-\zeta_j^{(n)})\not=0$, which is different from \cite[Remark 5.1]{DLL2021}. However, we have $W_j^{(n)}(\zeta_i^{(n)})-R_j^{(n)}(\zeta_i^{(n)})\sin^2(\zeta_i^{(n)}-\zeta_j^{(n)})\to0$ as $n\to\infty$. 	 
\end{remark}

\begin{table} 
	\caption{The parameterized boundary curves.}
	\label{boundary}
	\begin{tabular}{lll}
		\toprule[1pt]
		Boundary type           & Parameterization\\
		\midrule  
		Apple-shaped   & $\gamma(t)=\displaystyle\frac{0.55(1+0.9\cos{t}+0.1\sin{2t})}{1+0.75\cos{t}}(\cos{t}, \sin{t}), \quad t\in [0,2\pi]$  
		\vspace{1.5ex}\\ 
		Peanut-shaped  & 
		$\gamma(t)=0.275\sqrt{3\cos^2{t}+1}(\cos{t},\sin{t}), \quad
		t\in[0,2\pi]$
		\vspace{1.5ex} \\ 
		Peach-shaped  & 
		$\gamma(t)=0.22(\cos^2{t}\sqrt{1-\sin{t}}+2)(\cos{t}, \sin{t}), \quad t\in[0,2\pi]$
		\vspace{1.5ex}\\ 
		Drop-shaped  & 
		$\displaystyle \gamma(t)=(2\sin{\frac{t}{2}}-1, -\sin{t}), \quad t\in[0,2\pi]$ 
		\vspace{1.5ex}\\
		Heart-shaped  & 
		$\displaystyle \gamma(t)=(\frac{3}{2}\sin{\frac{3t}{2}}, \sin{t}), \quad t\in[0,2\pi]$\\
		\bottomrule[1pt]
	\end{tabular}
\end{table}

\subsection{Single-single layer potential formulation}

As a comparison, we introduce a computationally more convenient boundary integral formulation by using the single- and single-layer potentials. Using the fundamental solutions to the Helmholtz and modified Helmholtz equations, we represent the solutions of \eqref{vHvM} in terms of two single-layer potentials 
\begin{align}\label{singlelayer}
	v_{\rm H}(x)=\int_\Gamma G_{\rm H}(x,y)g_1(y)\mathrm{d}s(y), \quad v_{\rm M}(x)=\int_\Gamma G_{\rm M}(x,y)g_2(y)\mathrm{d}s(y),\quad x\in\mathbb{R}^2\setminus\Gamma_D,
\end{align}
where $g_1$ and $g_2$ are densities.

Using the boundary condition \eqref{bcv} and the jump relation of the single-layer potentials \eqref{singlelayer}, we deduce the coupled boundary integral equations
\begin{align}
	\begin{split}\label{boundaryIE}
		-u^{\rm inc}(x)&=\int_{\Gamma}G_{\rm H}(x,y)g_1(y)\mathrm{d}s(y)+ \int_{\Gamma}G_{\rm M}(x,y)g_2(y)\mathrm{d}s(y),\\
		-\frac{\partial u^{\rm inc}(x)}{\partial\nu(x)}&=-\frac{1}{2}g_1(x)+\int_{\Gamma}\frac{\partial G_{\rm H}(x,y)}{\partial\nu(x)}g_1(y)\mathrm{d}s(y)\\
		&\quad-\frac{1}{2}g_2(x)+\int_{\Gamma}\frac{\partial G_{\rm M}(x,y)}{\partial\nu(x)}g_2(y)\mathrm{d}s(y).
	\end{split}
\end{align}
Similarly, we may obtain the following parametric form of \eqref{boundaryIE}: 
\begin{align*}
	\mathcal{A}\psi\overset{\text{def}}{=}\left[     
	\begin{array}{cc}
		S^{\rm H} & S^{\rm M} \\ 
		-I+K^{\rm H}    & -I+K^{\rm M}
	\end{array}
	\right]\left[                  
	\begin{array}{c}
		\psi_1 \\ 
		\psi_2
	\end{array}
	\right]=\left[                  
	\begin{array}{c}
		\eta_1 \\ 
		\eta_2
	\end{array}
	\right],
\end{align*}
where $\eta_1=2(f_1\circ \gamma)$, $\eta_2=2(f_2\circ \gamma)|\gamma'|$, $\psi_j=(g_j\circ \gamma)|\gamma'|$, $j=1,
2$, and 
\begin{align*}
	(S^\sigma\phi)(t)&=\int_0^{2\pi}s^\sigma(t,\zeta)
	\phi(\zeta)\mathrm{d}\zeta,\quad 
	(K^\sigma\phi)(t)=\int_0^{2\pi}k^\sigma(t,\zeta)\phi(\zeta)\mathrm{d}\zeta
\end{align*}
with the kernels being given by
\begin{align*}
	s^\sigma(t,\zeta)&=\frac{\mathrm{i}}{2}H_0^{(1)}(\kappa_\sigma|\gamma(t)-\gamma(\zeta)|),\\
	k^\sigma(t,\zeta)&=\frac{\mathrm{i}\kappa_\sigma}{2}n(t)\cdot[\gamma(\zeta)-\gamma(t)] \frac{H_1^{(1)}(\kappa_\sigma|\gamma(t)-\gamma(\zeta)|)}{|\gamma(t)-\gamma(\zeta)|}.
\end{align*}
Here $\sigma={\rm H}$ or ${\rm M}$, $\kappa_{\rm H}=\kappa$, and $\kappa_{\rm M}={\rm i}\kappa$. The kernels $s^\sigma(t,\zeta)$ and $k^\sigma(t,\zeta)$ can be decomposed into 
\begin{align*}
	s^\sigma(t,\zeta)&= s_1^\sigma(t,
	\zeta)\ln\Big(4\sin^2\frac{t-\zeta}{2}\Big)+s_2^\sigma(t,\zeta),\\
	k^\sigma(t,\zeta)&= k_1^\sigma(t,
	\zeta)\ln\Big(4\sin^2\frac{t-\zeta}{2}\Big)+k_2^\sigma(t,\zeta),
\end{align*}
where the functions 
\begin{align*}
	s_1^\sigma(t,\zeta)&= -\frac{1}{2\pi}J_0(\kappa_\sigma|\gamma(t)-\gamma(\zeta)|), \\
	s_2^\sigma(t,\zeta)&=m^\sigma(t,\zeta)-m_1^\sigma(t,
	\zeta)\ln\Big(4\sin^2\frac{t-\zeta}{2}\Big),\\
	k_1^\sigma(t,\zeta)&= \frac{\kappa_\sigma}{2\pi}n(t)\cdot\big[
	\gamma(t)-\gamma(\zeta)\big]\frac{J_1(\kappa_\sigma|\gamma(t)-\gamma(\zeta)|)}{
		|\gamma(t)-\gamma(\zeta)|}, \\
	k_2^\sigma(t,\zeta)&=k^\sigma(t,\zeta)-k_1^\sigma(t,
	\zeta)\ln\Big(4\sin^2\frac{t-\zeta}{2}\Big)
\end{align*}
are analytic and their diagonal entries are given by 
\begin{align*}
	s_1^\sigma(t,t)=-\frac{1}{2\pi},\quad   s_2^\sigma(t,t)=
	\frac{\mathrm{i}}{2}-\frac{C}{\pi}-\frac{1}{\pi}\ln\big(\frac{\kappa_\sigma}{2}|\gamma'(t)|\big),
\end{align*}	
and
\begin{align*}
	k_1^\sigma(t,t)=0, \quad k_2^\sigma(t,t)=
	\frac{1}{2\pi}\frac{n(t)\cdot \gamma''(t)}{|\gamma'(t)|^2}.
\end{align*}

Compared with \eqref{cboundaryIE}, the coupled system \eqref{boundaryIE} is simpler computationally since it does not involve any hypersingular integral kernel. Theoretically, it is unclear whether the boundary integral equations \eqref{boundaryIE} and the corresponding parametric form have a unique solution. The analysis will be left for a future work.

\subsection{Numerical examples}

In this section, we present some numerical examples of both smooth and nonsmooth cavities to demonstrate the performance of the proposed methods. The code is written in Matlab using double precision arithmetic and the computations are run on a personal computer with a 64 GB RAM, 3.70 GHz Intel core i9 processor. 

\begin{table}
	\centering 
	\caption{Numerical errors of the wave fields for the apple- and peanut-shaped cavities by using the first approach for the double-single layer potential formulation.} 
	\label{numerror2} 
	\begin{tabular}{c|c|c|c|c|c|c}  
		\toprule[1pt]
		& \multicolumn{3}{c|}{Apple-shaped} & \multicolumn{3}{c}{Peanut-shaped}  \\ 
		\cline{2-7}
		$n$&$\|v_{\rm H}^*-v_{\rm H}^{(n)}\|$    &$\|v_{\rm M}^*-v_{\rm M}^{(n)}\|$
		&time&$\|v_{\rm H}^*-v_{\rm H}^{(n)}\|$ &$\|v_{\rm M}^*-v_{\rm M}^{(n)}\|$ &time\\
		\hline
		8&9.1028e-03&4.9959e-03&0.008s &6.4240e-03&3.3783e-03&0.006s\\
		16&1.9452e-03&7.4528e-04&0.01s&2.6370e-04&7.9816e-05&0.01s \\
		32&8.0695e-08&4.8943e-08&0.03s&5.1058e-08&1.5750e-08&0.02s \\
		64&2.2710e-12&6.4901e-13&0.08s&7.9013e-14&2.7135e-14&0.07s \\
		128&1.0169e-12&4.4399e-13&0.16s&1.1928e-12&3.8851e-13&0.15s \\
		\bottomrule[1pt] 
	\end{tabular}
\end{table}

\begin{table}
	\centering 
	\caption{Numerical errors of the far-field patterns for the apple- and peanut-shaped cavities by using the first approach for the double-single layer potential formulation.} 
	\label{numerrorfarfield} 
	\begin{tabular}{c|c|c|c|c}  
		\toprule[1pt]
		& \multicolumn{2}{c|}{Apple-shaped} & \multicolumn{2}{c}{Peanut-shaped}  \\ 
		\cline{2-5}
		$n$&$\|v_\infty^{(n_*)}-v_\infty^{(n)}\|$   
		&time&$\|v_\infty^{(n_*)}-v_\infty^{(n)}\|$  &time\\
		\hline
		8&2.9366e-02&0.007s &1.7153e-03&0.004s\\
		16&8.0457e-04&0.009s&5.8490e-07&0.007s \\
		32&2.2996e-07&0.02s&5.0560e-10&0.02s \\
		64&4.4398e-10&0.08s&5.0548e-10&0.07s \\
		128&4.4586e-10&0.16s&5.0677e-10&0.14s \\
		\bottomrule[1pt] 
	\end{tabular}
\end{table}

\begin{figure}
	\centering 
	\subfigure[$\Re(v^{(n_*)}_\infty)$ and $\Re(v^{(n)}_\infty)$]
	{\includegraphics[width=0.48\textwidth]{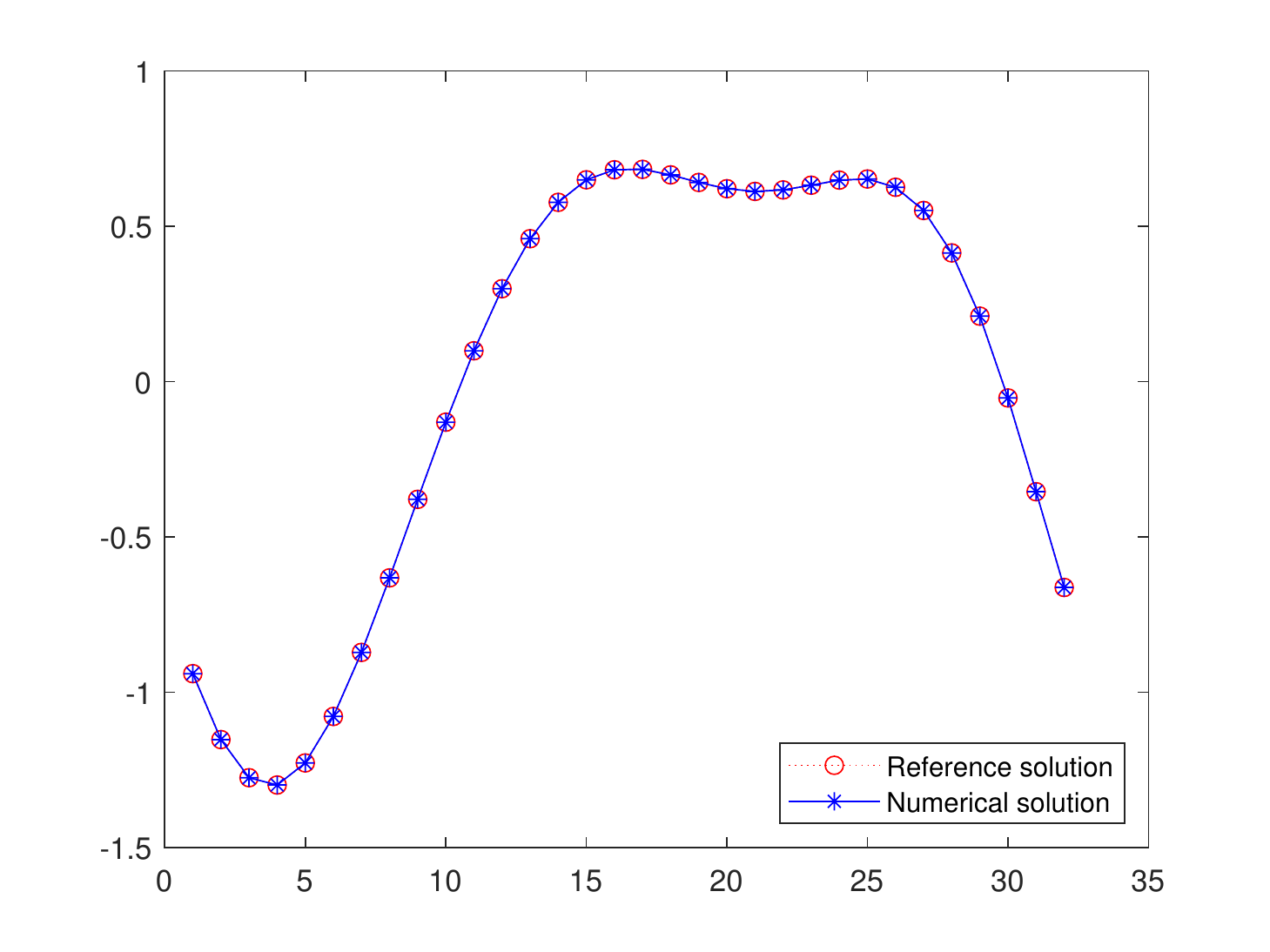}}
	\subfigure[$\Im(v^{(n_*)}_\infty)$ and $\Im(v^{(n)}_\infty)$]
	{\includegraphics[width=0.48\textwidth]{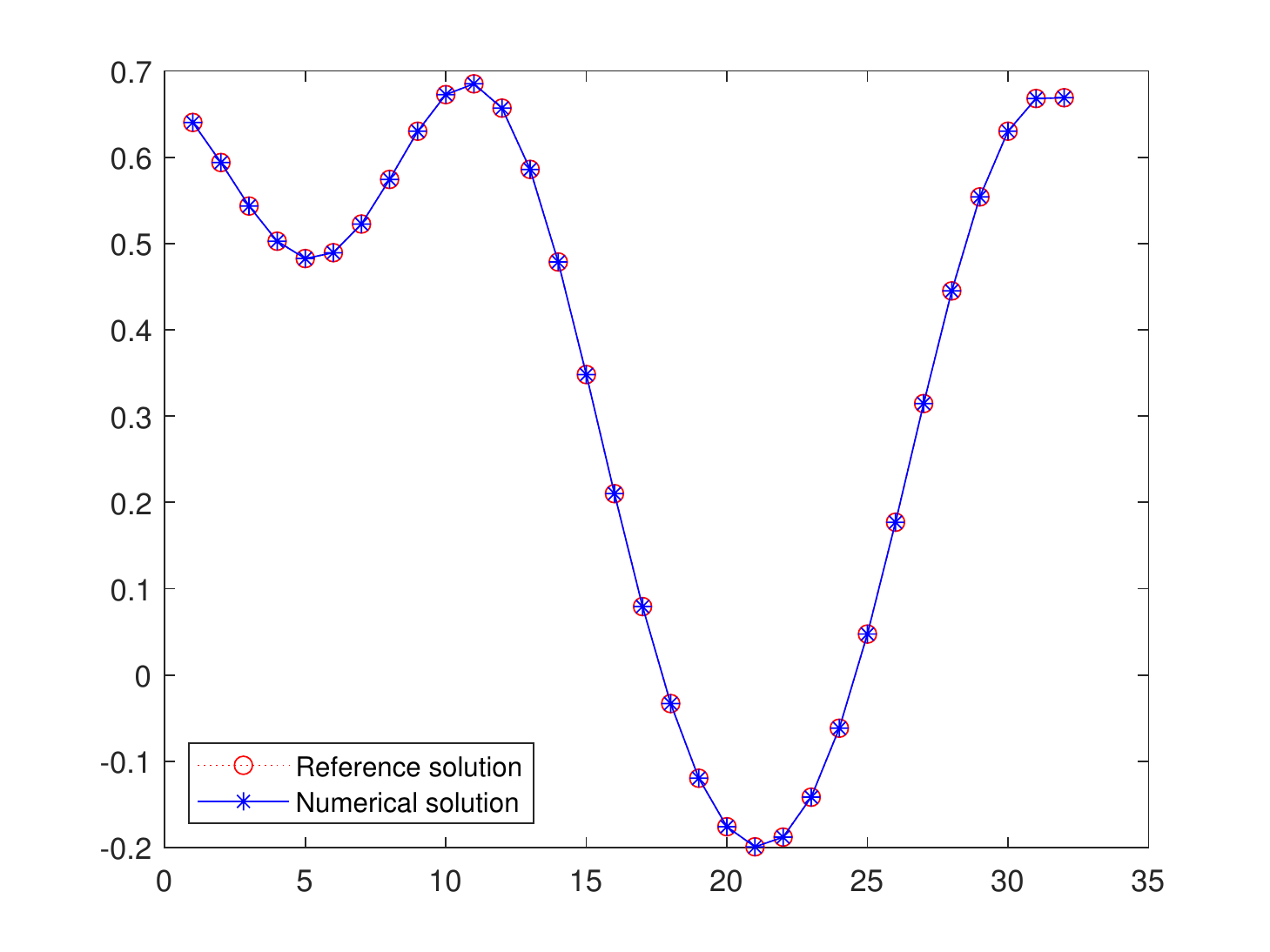}} 
	\caption{The reference and numerical solutions for the apple-shaped cavity: (a) the real part of the far-field pattern; (b) the imaginary part of the far-field pattern.}\label{farsolutions}
\end{figure}

\begin{table}
	\centering 
	\caption{Numerical errors of the wave fields for the apple- and peanut-shaped cavities by using the second approach for the double-single layer potential formulation.} 
	\label{numerror3} 
	\begin{tabular}{c|c|c|c|c|c|c}  
		\toprule[1pt]
		& \multicolumn{3}{c|}{Apple-shaped} & \multicolumn{3}{c}{Peanut-shaped}  \\ 
		\cline{2-7}
		$n$&$\|v_{\rm H}^*-v_{\rm H}^{(n)}\|$    &$\|v_{\rm M}^*-v_{\rm M}^{(n)}\|$
		&time&$\|v_{\rm H}^*-v_{\rm H}^{(n)}\|$ &$\|v_{\rm M}^*-v_{\rm M}^{(n)}\|$ &time\\
		\hline
		8&3.6851e-02&2.0239e-02&0.006s &1.1121e-02&6.5680e-03&0.007s\\
		16&1.7791e-03&1.2935e-03&0.008s&2.6716e-04& 8.4432e-05&0.01s \\
		32&3.4434e-07&9.9280e-08&0.02s&5.1173e-08& 1.5776e-08&0.02s \\
		64& 2.2717e-12&6.4932e-13&0.08s&7.8660e-14&2.7224e-14&0.08s \\
		128&1.0171e-12&4.4408e-13&0.17s&1.1929e-12&3.8858e-13&0.16s \\
		\bottomrule[1pt] 
	\end{tabular}
\end{table}

\begin{table}
	\centering 
	\caption{Numerical errors of the wave fields for the apple- and peach-shaped cavities by using the single-single layer potential formulation.} 
	\label{numerror4} 
	\begin{tabular}{c|c|c|c|c|c|c}  
		\toprule[1pt]
		& \multicolumn{3}{c|}{Apple-shaped} & \multicolumn{3}{c}{Peach-shaped}  \\ 
		\cline{2-7}
		$n$&$\|v_{\rm H}^*-v_{\rm H}^{(n)}\|$    &$\|v_{\rm M}^*-v_{\rm M}^{(n)}\|$
		&time&$\|v_{\rm H}^*-v_{\rm H}^{(n)}\|$ &$\|v_{\rm M}^*-v_{\rm M}^{(n)}\|$ &time\\
		\hline
		8&1.5871e-02&7.9131e-03&0.008s &2.2946e-03&1.1103e-03&0.008s\\
		16&1.6863e-02&9.6869e-03&0.01s&3.1410e-04&1.1890e-04&0.01s \\
		32&5.9174e-10&4.5291e-10&0.02s&3.9650e-05&1.4962e-05&0.02s \\
		64&7.1981e-15&5.9586e-15&0.06s&4.9690e-06&1.8742e-06&0.06s \\
		128&1.1133e-15&3.0029e-16&0.13s&6.2163e-07&2.3443e-07&0.12s \\
		\bottomrule[1pt] 
	\end{tabular}
\end{table}

\begin{table}
	\centering 
	\caption{Numerical errors of the far-field patterns for the apple- and peach-shaped cavities by using the 
		single-single layer potential formulation.} 
	\label{numerrorfarfield1} 
	\begin{tabular}{c|c|c|c|c}  
		\toprule[1pt]
		& \multicolumn{2}{c|}{Apple-shaped} & \multicolumn{2}{c}{Peach-shaped}  \\ 
		\cline{2-5}
		$n$&$\|v_\infty^{(n_*)}-v_\infty^{(n)}\|$   
		&time&$\|v_\infty^{(n_*)}-v_\infty^{(n)}\|$  &time\\
		\hline
		8&3.3854e-03&0.003s &5.4010e-04&0.003s\\
		16&1.0286e-03&0.005s&1.4150e-05&0.005s \\
		32&4.7122e-10&0.01s&1.6853e-06&0.01s \\
		64&1.8565e-14&0.04s&2.0524e-07&0.05s \\
		128&1.4716e-14&0.11s&2.5289e-08&0.11s \\
		\bottomrule[1pt] 
	\end{tabular}
\end{table}
\begin{figure}
	\centering 
	\subfigure[$\Re(v^{(n_*)}_\infty)$ and $\Re(v^{(n)}_\infty)$]
	{\includegraphics[width=0.48\textwidth]{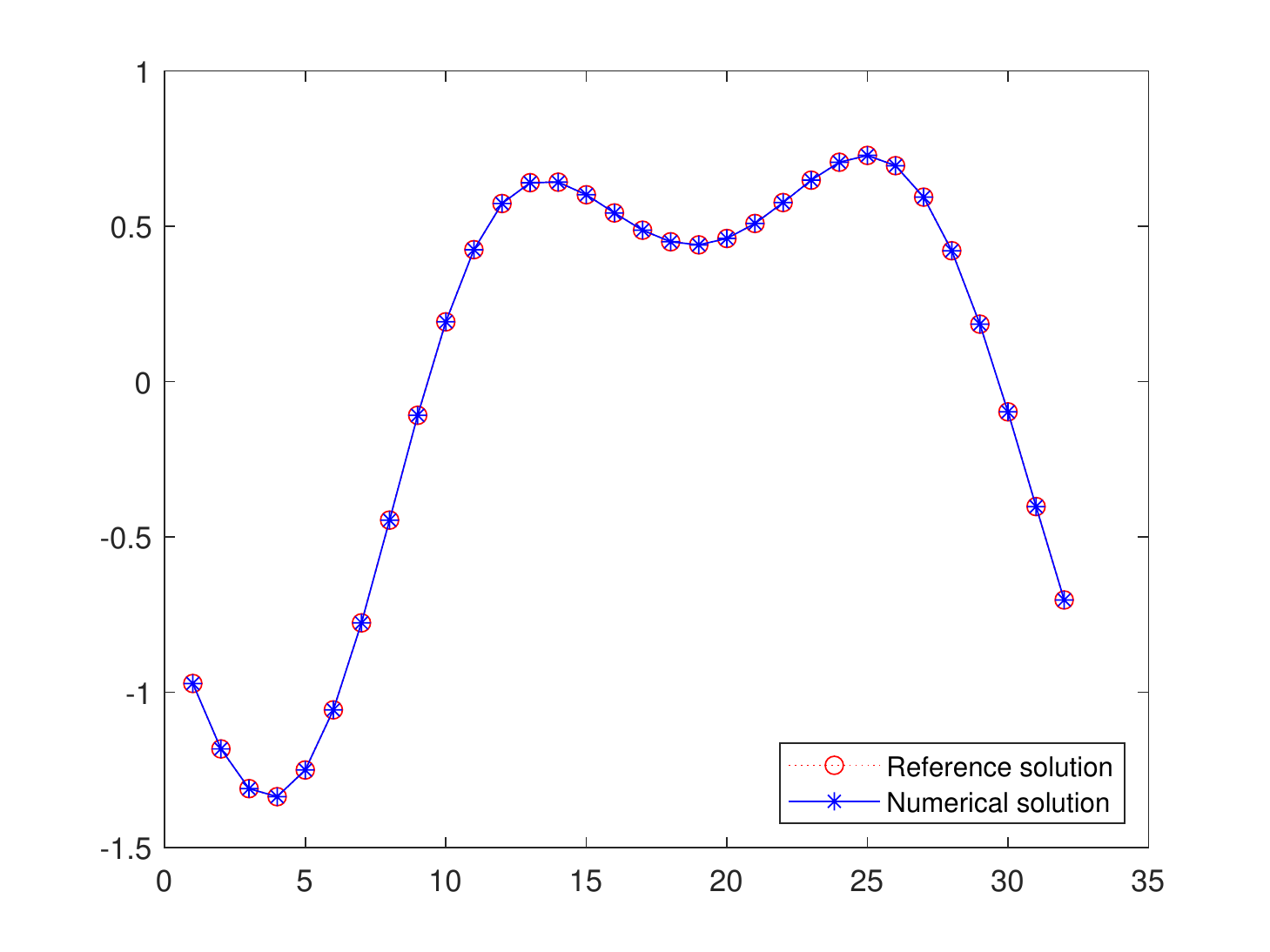}}
	\subfigure[$\Im(v^{(n_*)}_\infty)$ and $\Im(v^{(n)}_\infty)$]
	{\includegraphics[width=0.48\textwidth]{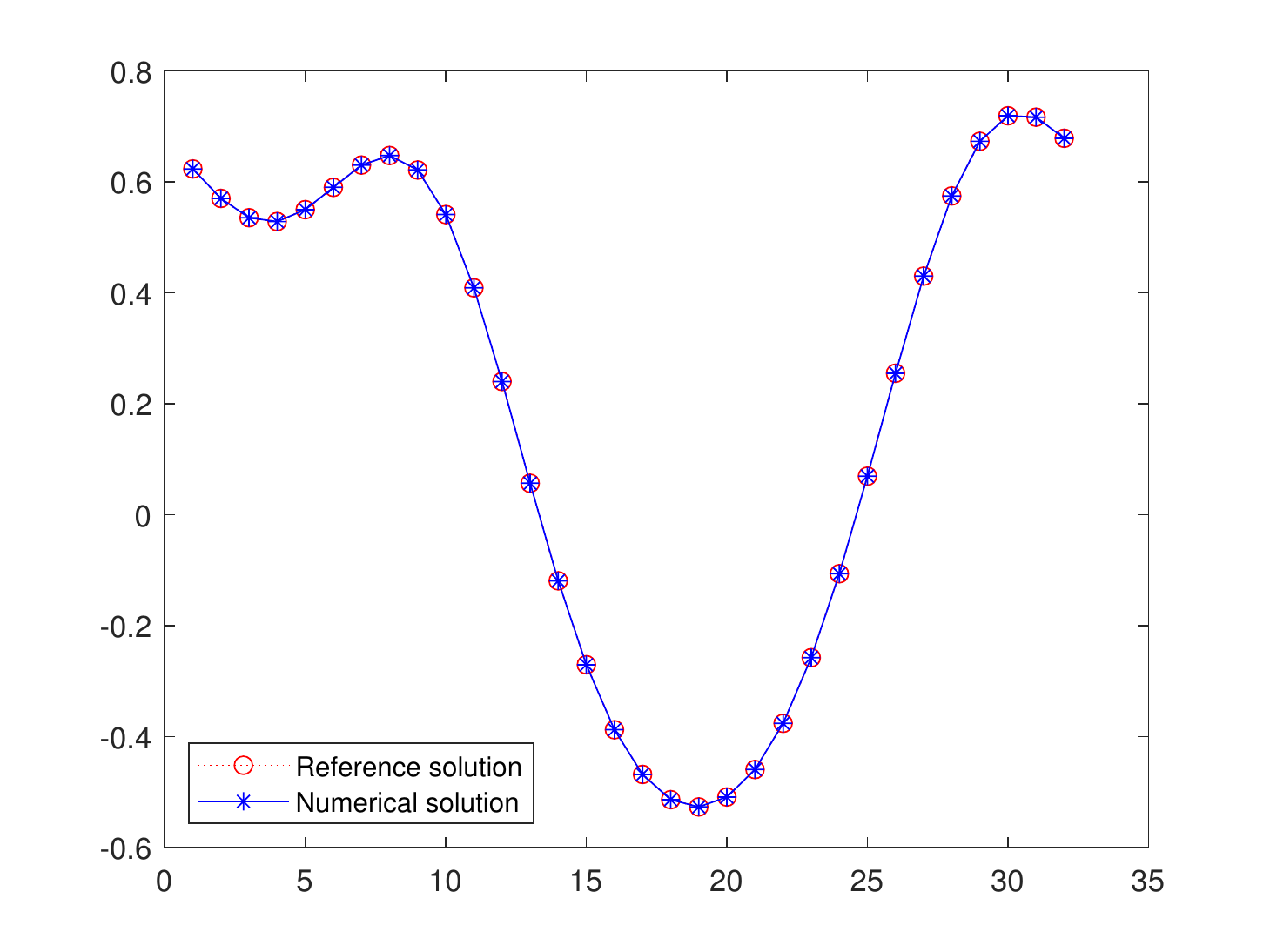}} 
	\caption{The reference and numerical solutions for the peach-shaped cavity: (a) the real part of the far-field pattern; (b) the imaginary part of the far-field pattern.}\label{farsolutions_peach}
\end{figure}

\subsubsection{Smooth cavities}

First, we consider the scattering by smooth cavities: the apple-, peanut-, and peach-shaped cavities. Table \ref{boundary} shows the parameterized equations for the boundary curves, where the appled- and peanut-shaped cavities have analytic boundary curves while the peach-shaped cavity has a $\mathcal{C}^2$ boundary curve.
These examples have been commonly used for benchmark tests of various scattering problems, e.g., the elastic obstacle scattering problem \cite{DLL2021}.

Let us begin with an example which has an analytic solution in order to test the accuracy of the methods. We consider a point source located at $\bar{x}=(0.1,0.2)^\top\in D$ and construct the corresponding exact solutions to the Helmholtz and modified Helmholtz equations 
\begin{align}\label{exact solution}
	v_{\rm H}^*(x)=H_0^{(1)}(\kappa|x-\bar{x}|),\quad v_{\rm M}^*(x)=H_0^{(1)}(\mathrm{i}\kappa|x-\bar{x}|),\quad x\in\mathbb{R}^2\setminus\overline{D},
\end{align}
which satisfy the Sommerfeld radiation condition \eqref{sc1}. Then the numerical solution is obtained by solving the boundary value problem \eqref{vHvM}--\eqref{sc1} with the following boundary conditions on $\Gamma$:
\begin{align*}
	\tilde f_1=v_{\rm H}^*+v_{\rm M}^*,\quad \tilde f_2=\partial_\nu v_{\rm H}^*+\partial_\nu v_{\rm M}^*. 
\end{align*}
Tables \ref{numerror2}, \ref{numerror3}, and \ref{numerror4} list the $L^2(\partial B_1)$ errors of the wave fields $v_{\rm H}$ and $v_{\rm M}$ at the wavenumber $\kappa=2$, where $\partial B_1=\{x\in\mathbb R^2: |x|=1\}$. Specifically, Tables \ref{numerror2} and \ref{numerror3} show the results by using the first and second approaches for the double-single layer potential formulation, respectively; Table \ref{numerror4} gives the results by using the single-single layer potential formulation. 

Now, we consider the scattering problem that the cavity is illuminated by a plane wave. Since there are no analytic solutions, we use the numerical solution with the number of collocation points $n_*=2048$ as the reference solution. We 
take the incident angle $\theta=\pi/6$, solve the boundary integral equations, and use \eqref{singlelayer_far} to  calculate the corresponding far-field patterns $v_\infty^{(n_*)}$ at 32 observation points which are uniformly distributed on the unit circle $\Omega$. To examine the convergence, the reference solutions are compared with the numerical solutions computed by using different numbers of collocation points. Table \ref{numerrorfarfield} shows the errors of the far-field patterns for the apple- and peanut-shaped cavities by using the first approach for the double-single layer potential formulation. Table \ref{numerrorfarfield1} shows the errors of the far-field patterns for the apple- and peach-shaped cavities by using the single-single layer potential formulation. Figures \ref{farsolutions} and \ref{farsolutions_peach} plot the real and imaginary parts of the reference solution $v_\infty^{(n_*)}$ and the numerical solution $v_\infty^{(n)}$ with the number of collocation points $n=32$ for the apple- and peach-shaped cavities, respectively. Since the results are similar for the second approach to the double-single layer potential formulation, they are not shown here for brevity. 

Based on the above numerical experiments, it can be observed that the errors decrease exponentially as the number of collocation points increases, which indicates the exponential convergence of the method and confirms our theoretical analysis. It can also be found that the convergence rate of the apple-shaped cavity is faster than that of the peach-shaped cavity, which is consistent with the fact that the apple-shaped cavity has an analytic boundary curve while the peach-shaped cavity has only a $\mathcal{C}^2$ boundary curve.

\subsubsection{Non-smooth cavities}

Next we report the scattering by nonsmooth cavities: the drop- and heart-shaped cavities. The parameterizations of the boundary curves of these two cavities are shown in Table \ref{boundary}. Clearly, the drop-shaped cavity is convex but the heart-shaped cavity is concave; both of the boundary curves have a single corner at $x_0\in\Gamma$ and are analytic on $\Gamma\setminus \{x_0\}$. The angle at the corner is defined as $\vartheta$ which satisfies $0<\vartheta<2\pi$. It can be seen from the parameterizations that the drop- and heart-shaped cavities have interior angles of $\vartheta=\pi/2$ and $\vartheta=3\pi/2$, respectively, and the corner point $x_0$ corresponds to the parameter $t=0$. 

Again, to test the accuracy, we consider the point source located at $\bar{x}=(0.1,0.2)^\top$ and $\bar{x}=(-0.5,0.2)^\top$ for the drop- and heart-shaped cavities, respectively, and construct the corresponding exact solution \eqref{exact solution}. Moreover, we employ the graded mesh in order to efficiently resolve the wave field near the corner, i.e., more collocation points are adopted on the boundary curve near the corner while less points are allocated at the part of the boundary curve that is far from the corner. Specifically, we take the substitution $t=w(s)$ in the parametric curve of the drop- and hear-shaped cavities \cite{DR-shu2,Kress1990}, where 
\begin{align*}
	w(s)=2\pi\frac{[v(s)]^p}{[v(s)]^p+[v(2\pi-s)]^p}, \qquad 0\leq s\leq 2\pi
\end{align*}
with
\begin{align*}
	v(s)=\Big(\frac{1}{p}-\frac{1}{2}\Big)\Big(\frac{\pi-s}{\pi}\Big)^3+\frac{1}{p}
	\frac{s-\pi}{\pi}+\frac{1}{2}, \quad p=2. 
\end{align*}
In the experiments, the collocation points are chosen as $s_j:=\pi j/n+\pi/(2n)$. The graded mesh of the collocation points $w(s_j)$, $j=0,\cdots,2n-1$ can be found in \cite[Figure 3]{DLL2021} for both of the drop- and heart-shaped boundary curves.

\begin{table}
	\centering 
	\caption{Numerical errors of the wave fields for the drop- and heart-shaped cavities by using the single-single layer potential formulation.} 
	\label{numerror5} 
	\begin{tabular}{c|c|c|c|c|c|c}  
		\toprule[1pt]
		& \multicolumn{3}{c|}{Drop-shaped} & \multicolumn{3}{c}{Heart-shaped}  \\ 
		\cline{2-7}
		$n$&$\|v_{\rm H}^*-v_{\rm H}^{(n)}\|$&$\|v_{\rm M}^*-v_{\rm M}^{(n)}\|$&time
		&$\|v_{\rm H}^*-v_{\rm H}^{(n)}\|$ &$\|v_{\rm M}^*-v_{\rm M}^{(n)}\|$&time \\
		\hline
		8&3.9875e-03&7.3325e-04& 0.003s&2.0080e-01&3.9261e-02& 0.003s \\
		16&1.2025e-05&2.2451e-06&0.007s&6.2198e-02&3.5015e-02& 0.007s \\
		32&1.5096e-06&1.8009e-07& 0.02s&1.5920e-03&1.0946e-03& 0.02s \\
		64&1.8909e-07&2.2602e-08& 0.07s&4.4875e-05&3.6727e-05& 0.07s \\
		128&2.3294e-08&2.7856e-09& 0.13s&3.2170e-09&2.5981e-09& 0.13s \\
		256&4.5217e-09&5.4078e-10& 0.67s&3.2103e-12&1.5088e-12& 0.68s \\
		512&2.1147e-10&2.5291e-11& 5.95s&6.0989e-13&3.3518e-13& 5.89s \\
		1024&1.3734e-11&1.6426e-12& 53.88s& 4.9250e-13&3.0598e-13& 53.90s \\
		\bottomrule[1pt] 
	\end{tabular}
\end{table}

\begin{table}
	\centering 
	\caption{Numerical errors of the far-field patterns for the drop- and heart-shaped cavities by using the 
		single-single layer potential formulation.} 
	\label{numerrorfarfieldcorner} 
	\begin{tabular}{c|c|c|c|c}  
		\toprule[1pt]
		& \multicolumn{2}{c|}{Drop-shaped} & \multicolumn{2}{c}{Heart-shaped}  \\ 
		\cline{2-5}
		$n$&$\|v_\infty^{(n_*)}-v_\infty^{(n)}\|$   
		&time&$\|v_\infty^{(n_*)}-v_\infty^{(n)}\|$  &time\\
		\hline
		8&8.5518e-02&0.003s &1.7155&0.003s\\
		16&2.4203e-03&0.008s&1.6035e-01&0.007s \\
		32&3.5585e-04&0.02s& 4.1046e-04&0.02s \\
		64&6.1901e-05&0.07s&4.1476e-06&0.07s \\
		128&6.7751e-06&0.13s& 3.9402e-10&0.12s \\
		256&1.1587e-05&0.66s&9.9265e-13&0.69s \\
		512&3.6241e-06&5.99s&2.0886e-13&5.99s \\
		1024&1.5637e-06&53.62s&2.0840e-13&54.15s\\
		\bottomrule[1pt] 
	\end{tabular}
\end{table}

\begin{figure}[hhh]
	\centering 
	\subfigure[$\Re(v^{(n_*)}_\infty)$ and $\Re(v^{(n)}_\infty)$]
	{\includegraphics[width=0.48\textwidth]{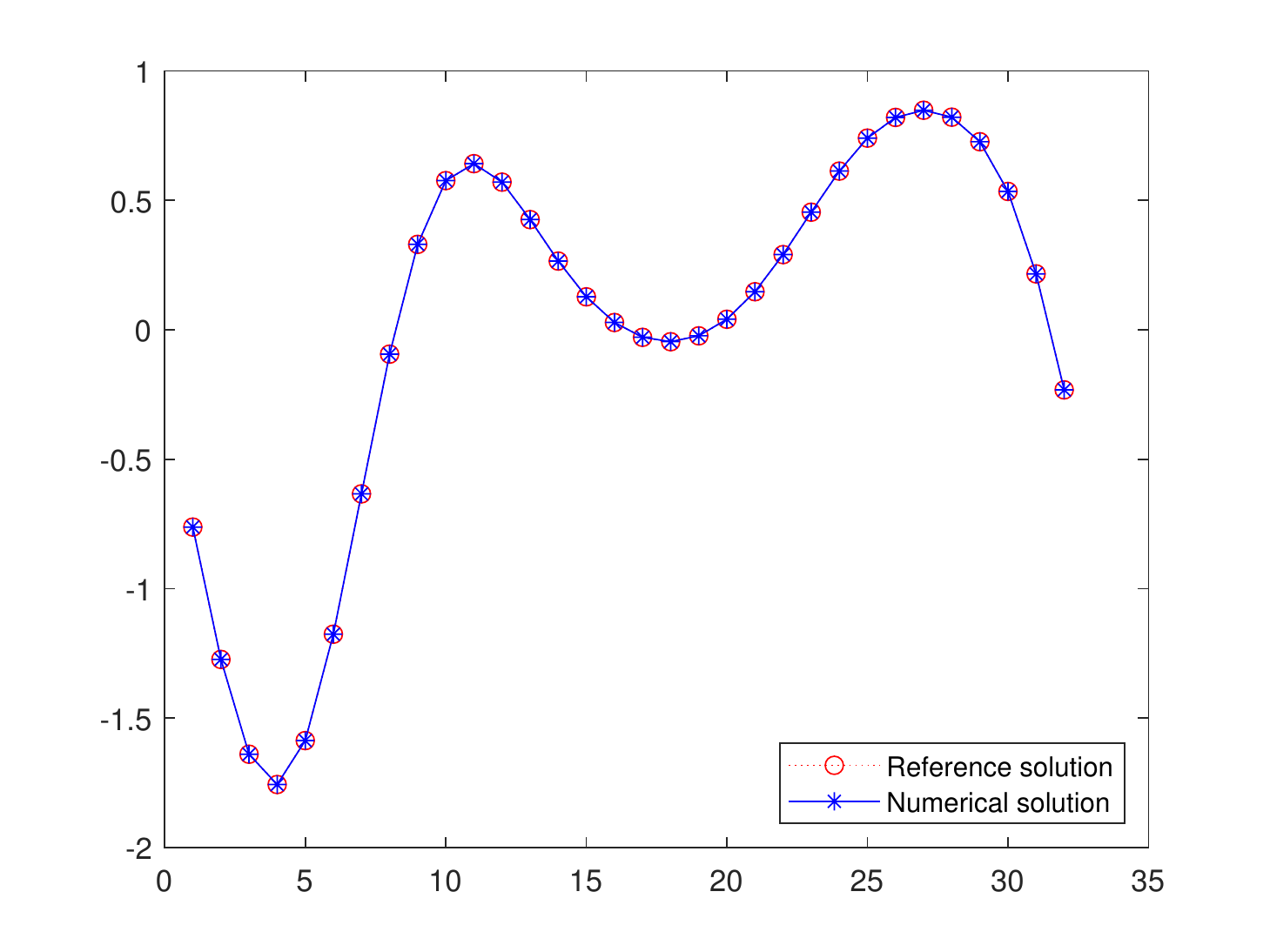}}
	\subfigure[$\Im(v^{(n_*)}_\infty)$ and $\Im(v^{(n)}_\infty)$]
	{\includegraphics[width=0.48\textwidth]{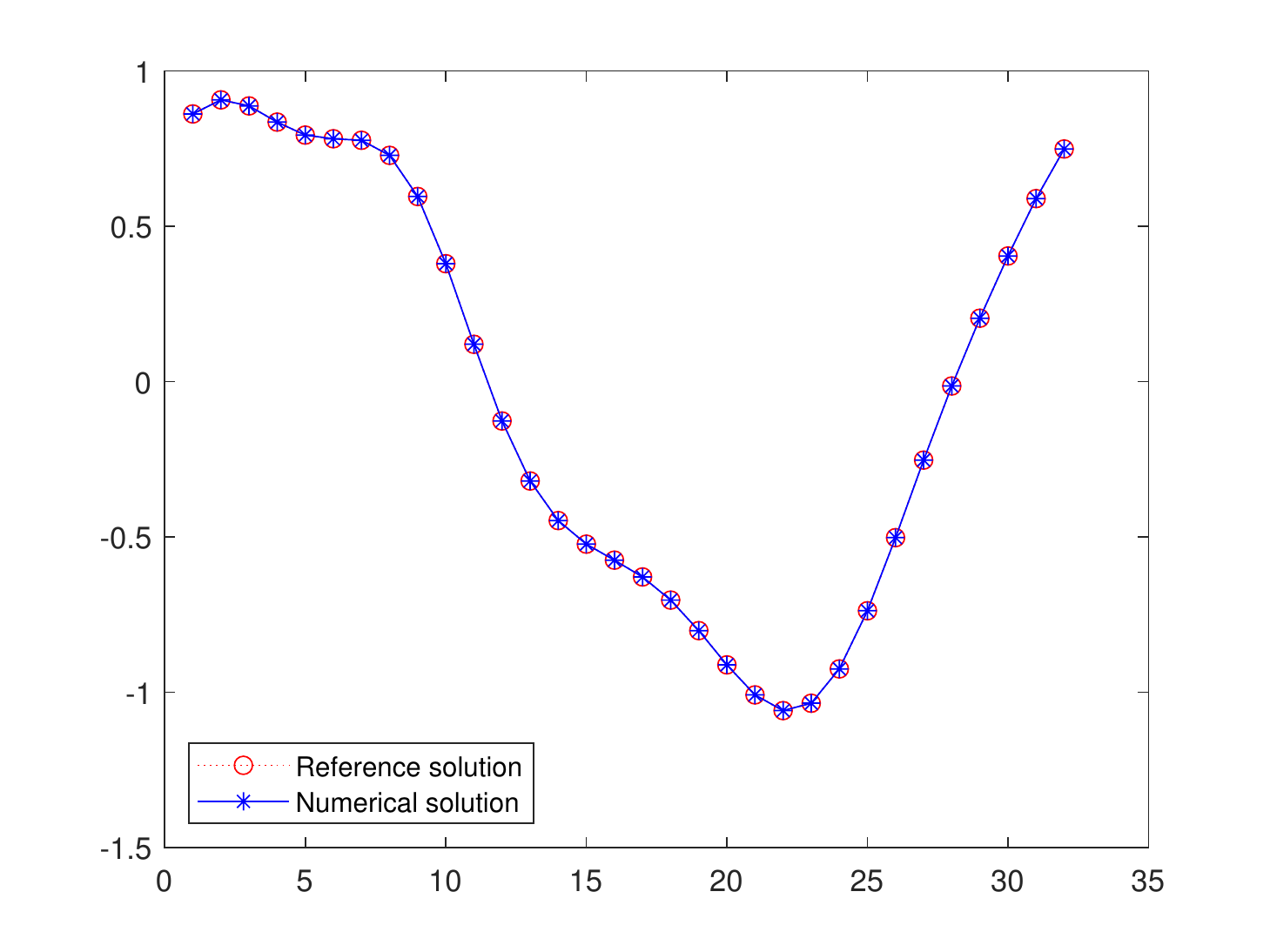}} 
	\caption{The reference and numerical solutions for the
		drop-shaped cavity: (a) the real part of the far-field pattern; (b) the imaginary part of the far-field pattern.}\label{farsolutions_drop}
\end{figure}

\begin{figure}[hhh]
	\centering 
	\subfigure[$\Re(v^{(n_*)}_\infty)$ and $\Re(v^{(n)}_\infty)$]
	{\includegraphics[width=0.48\textwidth]{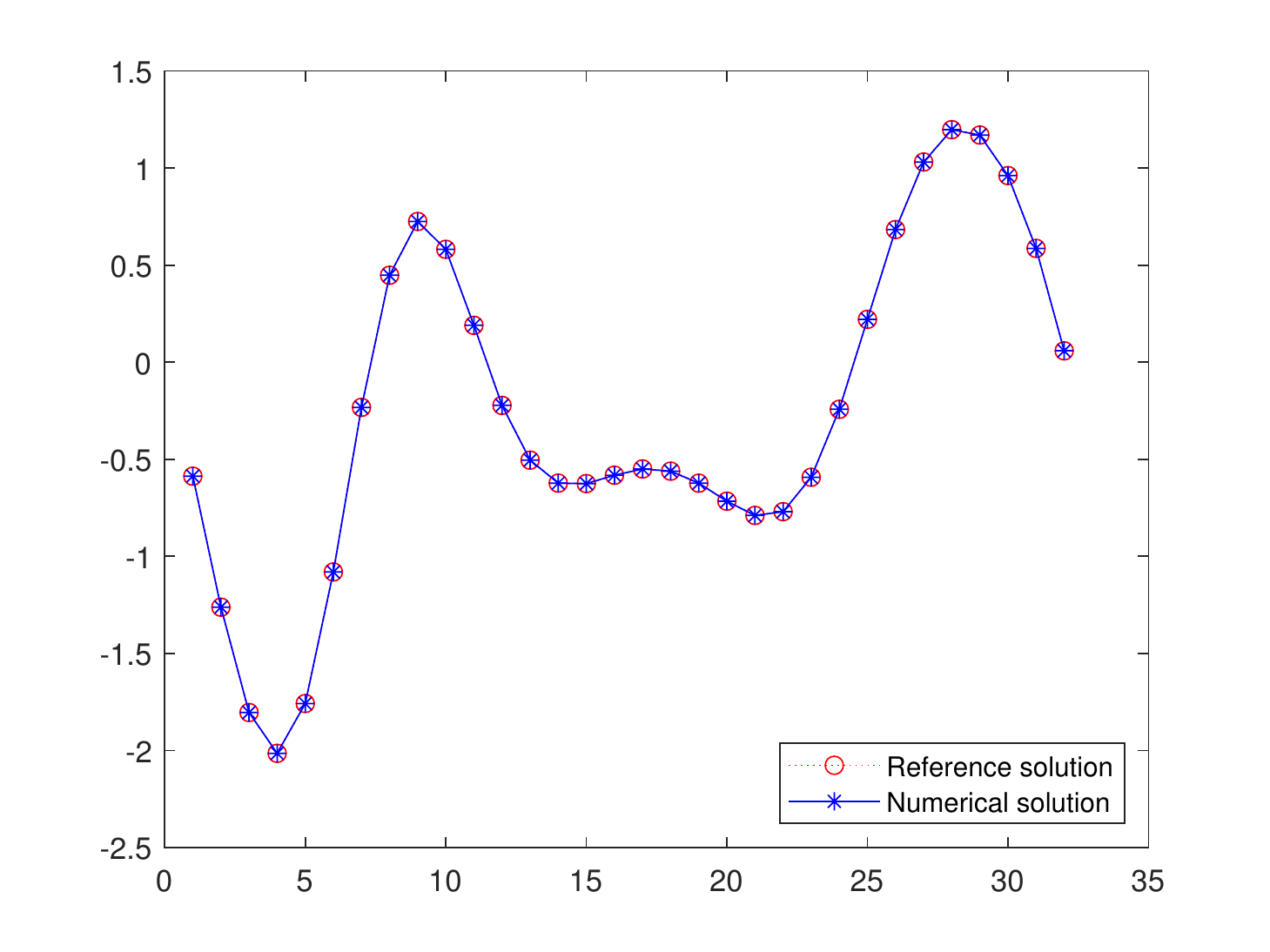}}
	\subfigure[$\Im(v^{(n_*)}_\infty)$ and $\Im(v^{(n)}_\infty)$]
	{\includegraphics[width=0.48\textwidth]{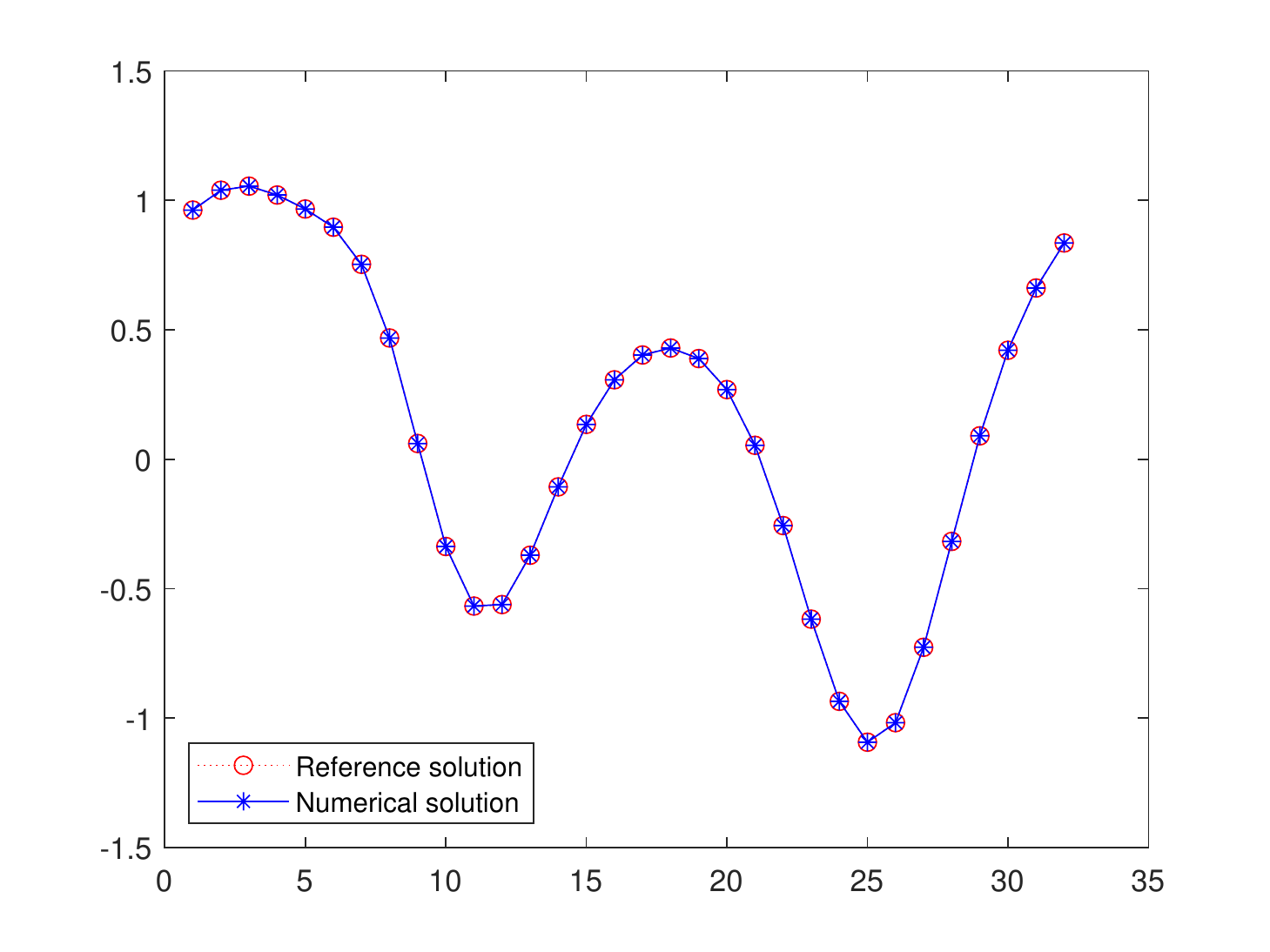}} 
	\caption{The reference and numerical solutions for the
		heart-shaped cavity: (a) the real part of the far-field pattern; (b) the imaginary part of the far-field pattern.}\label{farsolutions_heart}
\end{figure}

Table \ref{numerror5} lists the $L^2(\partial B_2)$ errors between the numerical solution and the exact solution
\eqref{exact solution} for the drop- and heart-shaped cavities, where $\partial B_2=\{x\in\mathbb R^2: |x|=2\}$. Clearly, the errors decay rapidly for both cases by using the graded meshes. This is due to the reason that the exact solution is analytic. On the other hand, the results are more accurate for the heart-shaped cavity than those of the drop-shaped cavity. We believe that this phenomenon is related to the geometry of the cavity. The analysis will be left for a future work. Similar to the scattering by smooth cavities, we consider the plane wave incidence with incident angle $\theta=\pi/6$
and compute the far-field patterns $v_\infty^{(n_*)}$ with $n_*=2048$ as reference solutions. Table \ref{numerrorfarfieldcorner} shows the errors of the far-field patterns by using the single-single layer potential formulation. The numerical solution $v_\infty^{(n)}$ and the reference solution $v_\infty^{(n_*)}$ are shown in Figures
\ref{farsolutions_drop} and \ref{farsolutions_heart} for the drop- and heart-shaped obstacles, respectively. It can be seen that the two solutions agree with each other very well.

\section{Conclusion}\label{s_c}

In this paper, we have studied the biharmonic wave scattering problem of a cavity embedded in an infinite thin plate. Based on the biharmonic wave operator splitting, the scattering problem of the fourth-order biharmonic wave equation is reduced into a coupled boundary value problem of the Helmholtz and modified Helmholtz equations. With the help of potential theory, a novel boundary integral formulation is proposed for the scattering problem. Based on an appropriate regularizer, the operator equation is split into an isomorphic operator plus a compact one. The well-posedness is established for the coupled system. The convergence analysis is carried out for both the semi- and full-discrete schemes by using the collocation method. To demonstrate the superior performance of the proposed method, numerical experiments are presented for both smooth and nonsmooth cavities. The numerical results show that the proposed method is highly accurate even for nonsmooth examples. 

In this work, we examined only the clamped boundary condition. At the edge of cavities, many other boundary conditions can be considered to take into account of different physical behavior \cite{GGS-10}. Clearly, a different boundary condition will lead to a different formulation of boundary integral equations. We are investigating these problems and will report the progress elsewhere in the future.

\end{document}